\documentclass[reqno]{amsart}
\usepackage{amsmath}
\usepackage{amsthm}
\usepackage{amsfonts}
\usepackage{amssymb}
\usepackage{url}

\usepackage{verbatim}
\usepackage{subfigure}

\newtheorem{theorem}{Theorem}[section]

\newtheorem{lemma}[theorem]{Lemma}

\newtheorem{remark}[theorem]{Remark}

\theoremstyle{definition}
\numberwithin{equation}{section}  

\usepackage{graphicx}

\usepackage[all]{xy}
\numberwithin{equation}{section}
\def\N{\mathbb N}

\def\n{\nonumber}
\def\d{\displaystyle}
\def\v{\tilde v}

\title[Moving-boundary 1-D compressible Euler equations  in vacuum]
{Well-posedness in smooth function spaces for the moving-boundary 
 1-D compressible Euler equations in physical vacuum}

\author[D. Coutand]{Daniel Coutand}
\author[S. Shkoller]{Steve Shkoller}
\address{\scshape{canpde}, Maxwell Institute for Mathematical Sciences and department of Mathematics, Heriot-Watt University, Edinburgh, EH14 4AS, UK}
\address{Department of Mathematics, University of California, Davis, CA 95616, USA}

\subjclass{35L65, 35L70, 35L80, 35Q35, 35R35, 76B03}
\keywords{compressible Euler equations,  gas dynamics, free boundary problems, physical vacuum, characteristic hyperbolic systems,
degenerate hyperbolic systems, systems of conservation laws}

\email{D.Coutand@ma.hw.ac.uk}
\email{shkoller@math.ucdavis.edu}

\begin{document}

\begin{abstract}
The free-boundary compressible 1-D Euler equations with moving {\it physical
vacuum} boundary are  a system of hyperbolic conservation laws which are
both {\it characteristic} and  {\it degenerate}.   The {\it physical vacuum singularity}
(or rate-of-degeneracy)  requires the sound speed $c= \gamma \rho^{ \gamma -1}$ to scale as the square-root of the distance to the
vacuum boundary, and has attracted a great deal of attention in recent years \cite{IMA2009}.  We establish the existence of unique solutions to this system on a short time-interval, which are smooth (in Sobolev
spaces) all the way to the moving boundary.   The proof is founded on a
new higher-order Hardy-type inequality in conjunction with an approximation of the Euler equations
consisting of a particular degenerate parabolic regularization.  Our regular solutions can be viewed as
{\it degenerate viscosity solutions}.
\end{abstract}

\maketitle
{\small 
\tableofcontents}

\section{Introduction}
\label{sec_introduction}

\subsection{Compressible Euler equations and the physical vacuum boundary} 
This paper is concerned with the evolving vacuum state in a compressible gas flow.
For $0 \le t \le T$,
the evolution of a one-dimensional  compressible gas  moving inside of a dynamic vacuum  boundary is modeled by the one-phase compressible Euler equations:
\begin{subequations}
  \label{ceuler}
\begin{alignat}{2}
\rho\left[u_t+ uu_x\right] + p(\rho)_x &=0  &&\text{in} \ \ I (t) \,, \label{ceuler.a}\\
\rho_t+ (\rho u)_x&=0 
&&\text{in} \ \ I (t) \,, \label{ceuler.b}\\
p &= 0 \ \ &&\text{on} \ \ \Gamma(t) \,, \label{ceuler.c}\\
\mathcal{V} (\Gamma(t))& = u  &&\ \ \label{ceuler.d}\\
(\rho,u)   &= (\rho_0,u_0 ) \ \  &&\text{on} \ \  I(0) \,, \label{ceuler.e}\\
   I (0) &= I =\{0<x<1\} \,.  && \label{ceuler.f}
\end{alignat}
\end{subequations}
The open, bounded interval
 $I (t) \subset \mathbb{R}  $ denotes the changing domain occupied by the gas,  $\Gamma(t):= \partial I (t)$ denotes
 the moving vacuum boundary, $ \mathcal{V} (\Gamma(t))$ denotes the
 velocity of $\Gamma(t)$.
  The scalar field $u $ denotes the Eulerian velocity
field, $p$ denotes the pressure function, and $\rho$ denotes the density of the gas.
The equation of state $p(\rho)$ is given by
\begin{equation}\label{eos}
p(x,t)= C_\gamma\, \rho(x,t)^\gamma\ \  \text{ for } \ \  \gamma> 1 ,
\end{equation} 
where $C_\gamma $ is the adiabatic constant which we set to one, and
$$
\rho>0 \ \text{ in } \ I (t) \ \ \ \text{ and } \ \ \ \rho=0 \ \text{ on } \Gamma(t) \,.
$$
Equation (\ref{ceuler.a}) is the conservation of momentum; (\ref{ceuler.b}) is the conservation of mass; the boundary
condition (\ref{ceuler.c}) states that  the pressure (and hence density)  must vanish along the vacuum boundary; (\ref{ceuler.d}) states that the
vacuum boundary is moving with  the fluid velocity, and (\ref{ceuler.e})-(\ref{ceuler.f}) are the 
initial conditions for the density, velocity, and domain.  Using the equation of state (\ref{eos}),
 (\ref{ceuler.a}) is written as
\begin{alignat}{2}
\rho [u_t+ uu_x]+  (\rho^{\gamma})_x &=0 \ \ \   &&\text{in} \ \ I (t) \,. \tag{ \ref{ceuler.a}'} 
\end{alignat}

With the sound speed given by $c^2(x,t)=  \gamma \rho^{\gamma-1}(x,t)$,  and with $c_0 = c(x,0)$, the condition
\begin{equation}\label{phys_vac}
0 < \left|\frac{ \partial c_0^2}{ \partial x}\right| < \infty  \text{ on } \Gamma
\end{equation} 
defines a {\it physical vacuum} boundary  (or singularity) (see  \cite{Liu1996}, \cite{LiYa1997}, \cite{LiYa2000},
 \cite{XuYa2005}). Since $ \rho_0 >0$ in $ I$, (\ref{phys_vac}) implies that for some positive constant $C$ and $x\in  I$ near the vacuum boundary $\Gamma$,
\begin{equation}\label{degen}
\rho_0^{\gamma-1}(x) \ge  C \text{dist}(x, \Gamma) \,.
\end{equation} 
Equivalently,  the physical vacuum condition (\ref{degen}) implies that for some $ \alpha >0$,
\begin{equation}\label{degen1}
\left|\frac{ \partial \rho_0^{\gamma-1}}{\partial x}(x)\right|\ge 1 \text{ for any $x$ satisfying $d(x,\partial I)\le\alpha$} \,,
\end{equation} 
and for a constant $C_ \alpha $, depending on $ \alpha $, 
\begin{equation}\label{degen2}
\rho_0^{\gamma-1}(x)\ge C_\alpha>0 \text{ for any $x$ such that $d(x,\partial I)\ge\alpha$} \,.
\end{equation}

\vspace{.05 in}

Because of condition (\ref{degen}), the compressible Euler system (\ref{ceuler}) is a {\it degenerate} and  {\it characteristic} hyperbolic
system to which standard methods of symmetric hyperbolic conservation laws cannot be applied in standard Sobolev spaces.   In
\cite{CoLiSh2009}, we established a priori estimates for the multi-D compressible Euler equations with physical vacuum boundary.

The main result of this paper is the construction of unique solutions in the 1-D case, which are smooth all the way to the moving vacuum boundary on
a (short) time-interval $[0,T]$, where $T$ depends on the initial data.  
  We combine
 the methodology of our a priori estimates \cite{CoLiSh2009}, with a particular degenerate parabolic regularization of Euler, which
follows our methodology in \cite{CoSh2006, CoSh2007}, as well as a  new  higher-order Hardy-type inequality which permits the
construction of solutions to this degenerate parabolic regularization.  As we describe below in Section \ref{history}, our solutions
can be thought of as {\it degenerate viscosity solutions}.  The  multi-D existence theory is treated in \cite{CoSh2009}.

\subsection{Fixing the domain and the Lagrangian variables on the reference interval $ I$}   We transform the system (\ref{ceuler}) into
Lagrangian variables.
We let $\eta(x,t)$ denote the ``position'' of the gas particle $x$ at time $t$.  Thus,
\begin{equation}
\nonumber
\begin{array}{c}
\partial_t \eta = u \circ \eta $ for $ t>0 $ and $
\eta(x,0)=x
\end{array}
\end{equation}
where $\circ $ denotes composition so that
$[u \circ \eta] (x,t):= u(\eta(x,t),t)\,.$  
We set 
\begin{align*}
v &= u \circ \eta   \text{ (Lagrangian velocity)},  \\
f&= \rho \circ \eta   \text{ (Lagrangian density)}. 
\end{align*}

The Lagrangian version of equations (\ref{ceuler.a})-(\ref{ceuler.b}) can be written on
the fixed reference domain $ I$ as
\begin{subequations}
\label{ceuler00}
\begin{alignat}{2}
f  v_t + ( f^ \gamma)_x &=0 \ \ && \text{ in } I  \times (0,T] \,, \label{ceuler00.a} \\
f _t + f v_x/\eta_x  &=0 \ \ && \text{ in } I  \times (0,T] \,,\label{ceuler00.b}  \\
f  &=0 \ \ && \text{ in } I  \times (0,T] \,,\label{ceuler00.c}  \\
(f,v,\eta)  &=(\rho_0, u_0, e) \ \  \ \ && \text{ in } I  \times \{t=0\} \,, \label{ceuler00.d} 
\end{alignat}
\end{subequations}
where $e(x)=x$ denotes the identity map on $I $.

It follows from solving the equation (\ref{ceuler00.b}) that
\begin{equation}\label{J}
f =\rho \circ \eta = \rho_0/\eta_x,
\end{equation}
so that the initial density function $\rho_0$ can be viewed as a parameter in the Euler equations.   Let $\Gamma:= \partial I $ denote
the initial vacuum boundary;  we write the compressible Euler equations (\ref{ceuler00}) as
\begin{subequations}
\label{ceuler0}
\begin{alignat}{2}
\rho_0 v_t +  (\rho_0^ \gamma /\eta_x^\gamma)_x &=0 \ \ && \text{ in } I  \times (0,T] \,, \label{ceuler0.a} \\
(\eta, v)  &=( e, u_0) \ \  \ \ && \text{ in } I  \times \{t=0\} \,, \label{ceuler0.b} \\
\rho_0^{\gamma-1}& = 0 \ \ &&\text{ on }  \Gamma \,, \label{ceuler0.c}  
\end{alignat}
\end{subequations}
with $ \rho _0^{ \gamma -1}(x) \ge C \operatorname{dist}( x, \Gamma ) $ for $x \in I$ near $\Gamma$.

\subsection{Setting $\gamma=2$}We will begin our analysis of  (\ref{ceuler0}) by considering the
case that $\gamma=2$, in which case,  we seek solutions $\eta$ to the following system:
\begin{subequations}
  \label{ce0}
\begin{alignat}{2}
\rho_0  v_t + (\rho_0^2/\eta_x^2)_x&=0  &&\text{in} \ \  I \times (0,T] 
\,, \label{ce0.a}\\
(\eta,v)&= (e,u_0 ) \ \ \ &&\text{on} \ \  I \times \{t=0\} \,, \label{ce0.b}\\
\rho_0& = 0 \ \ &&\text{ on }  \Gamma \,, \label{ce0.c}
\end{alignat}
\end{subequations}
with $ \rho _0(x) \ge C \operatorname{dist}( x, \Gamma ) $ for $x \in I$ near $\Gamma$.

The equation
(\ref{ce0.a}) is equivalent to 
\begin{equation}
v_t + 2 \eta_x ^{-1}  (\rho_0 \eta_x^{-1} )_x=0 \label{ce_vor} \,,
\end{equation} 
and (\ref{ce_vor}) can be written as
 
\begin{equation}
v_t +\rho_0 ( \eta_x^{-2})_x + 2(\rho_0)_x \eta_x^{-2}  =0 \label{ce_elliptic} \,.
\end{equation} 
Because of the degeneracy caused by $\rho_0 =0$ on $\Gamma$, all three equivalent forms of the compressible
Euler equations are crucially used in our analysis.  The equation (\ref{ce0.a}) is used for energy estimates, while (\ref{ce_vor}) 
 and (\ref{ce_elliptic}) are used for additional elliptic-type estimates that rely on our higher-order Hardy-type inequality.

\subsection{The reference domain $ I$}\label{subsec_domain}
As we have already noted above, 
 the initial domain $ I \subset \mathbb{R} $ at time $t=0$ is given by
$$
 I = (0,1)\,,
$$
and the initial boundary points are given by $\Gamma = \partial I = \{0,1\}$.

\subsection{The higher-order energy function for the case $\gamma=2$}
We will consider the following higher-order energy function:
\begin{align}
E(t,v) & =\sum_{s=0}^4 \|\partial_t^sv(t,\cdot)\|^2_{H^{2-\frac{s}{2}}(I)} + \sum_{s=0}^2 \|\rho_0 \partial_t^{2s} v(t,\cdot)\|^2_{H^{3-s}(I)} \n \\
& \qquad \qquad +\|\sqrt{\rho_0}\partial_t \partial_x^{2} v(t,\cdot)\|^2_{L^2(I)} +\|\sqrt{\rho_0}\partial_t^{3}\partial_x v(t,\cdot)\|^2_{L^2(I)}\,.
\label{E}
\end{align}

We define the polynomial function $M_0$ by
\begin{equation}\label{M0}
M_0 = P(E(0,v)) \,.
\end{equation}

 \subsection{The Main Result}   
\begin{theorem} [Existence and uniqueness for the case $\gamma=2$] \label{theorem_main}
 Given initial data $(u_0, \rho_0)$ such that $M_0< \infty $ and the physical vacuum condition (\ref{degen}) holds for $\rho_0$, there exists a  solution to (\ref{ce0}) (and hence (\ref{ceuler})) on $[0,T]$ for $T>0$ taken
 sufficiently small, such that
 $$
 \sup_{t \in [0,T]} E(t) \le 2M_0 \,.
 $$
 Moreover if the initial data satisfies
 \begin{equation}\label{uniquedata}
 \sum_{s=0}^3 \|\partial_t^sv(0,\cdot)\|^2_{{3-s}} + \sum_{s=0}^3 \|\rho_0 \partial_t^{2s} v(0,\cdot)\|^2_{{4-s}} < \infty \,,
 \end{equation} 
 then the solution is unique.
\end{theorem} 
\begin{remark} The case of arbitrary $\gamma >1$ is treated in Theorem \ref{thm_main2} below.
\end{remark} 
\begin{remark} Given the regularity provided by the energy function (\ref{E}), we see that the Lagrangian flow map
$\eta \in C([0,T], H^2(I))$.   In our estimates for the multi-D problem \cite{CoLiSh2009}, we showed that $\eta$ gains
regularity with respect to the velocity field $v$, and this fact is essential to control the geometry of the evolving free-surface.
This improved regularity for $\eta$ also holds in the 1-D setting, but it is not necessary for our estimates, as no
geometry is involved.
\end{remark} 
\begin{remark}
Given $u_0$ and $\rho_0$,  and using the fact that $\eta(x,0)=x$, the quantity $v_t|_{t=0}$ is computed using (\ref{ce0.a}):
$$
v_t|_{{t=0}} =- \left.\left({\frac{1}{\rho_0}} (\rho_0^2/\eta_x^2)_x \right)\right|_{t=0} =     - 2\frac{\partial \rho_0}{\partial x}  \,.
$$
Similarly, for all $k \in \mathbb{N}  $,
$$
\partial_t^k v|_{{t=0}} =- \left. \left( {\frac{1}{\rho_0}} (\rho_0^2/\eta_x^2)_x  \right)\right|_{t=0} \,,
$$
so that  each $\partial_t^k v|_{t=0}$ is a function of space-derivates of $u_0$ and $\rho_0$.
\end{remark} 

\begin{remark}Notice that our functional framework provides solutions which have optimal Sobolev regularity all the way
to the boundary.     Hence, in the physical case that $c \sim \sqrt{\text{dist}(\partial I)}$,  no singular behavior  occurs near
the vacuum boundary, even though both families of characteristics  cross, and in particular, meet tangentially to $\Gamma(t)$ at a point.
\end{remark} 

\begin{remark}  Because of the degeneracy of the density function $\rho_0$ at the initial boundary $ \partial I$, no compatibility
conditions are required of the initial data.
\end{remark}

\subsection{History of prior results for the compressible Euler equations with vacuum boundary}\label{history}
Some of the early developments in the theory of vacuum states for  compressible gas dynamics can be found in \cite{LiSm1980} and
\cite{Lin1987}.
We are aware of only a handful of previous theorems pertaining to the existence of solutions to the compressible and {\it undamped} Euler equations with moving  vacuum boundary.   Makino
 \cite{M1986} considered compactly supported initial data, and treated the compressible  Euler equations for a gas as being set on  $\mathbb{R}^3  \times (0,T]$.  With
 his methodology, it is not possible to track the location of the vacuum boundary (nor is it necessary); nevertheless, an existence theory was developed in
 this context, by a variable change that permitted the standard theory of symmetric hyperbolic systems to be employed.  Unfortunately, the constraints on the data are too severe to allow for the evolution of the physical vacuum boundary.
 
 In \cite{Li2005b}, Lindblad proved existence and uniqueness for the 3-D compressible Euler equations modeling a {\it liquid} rather than a gas.
 For a compressible liquid, the density $\rho\ge \lambda>0$ is assumed to be a strictly positive constant on the moving vacuum boundary $\Gamma(t)$ and is
 thus uniformly bounded below  by a positive constant.  As such, the compressible liquid provides a uniformly hyperbolic, but characteristic, system.  Lindblad used Lagrangian variables combined with Nash-Moser iteration to construct solutions.   More recently,  Trakhinin \cite{Tr2008}
provided an alternative proof for the existence of a compressible liquid, employing a solution strategy based on symmetric hyperbolic systems
combined with Nash-Moser iteration.  

In the presence of damping, and with mild singularity, some existence results of smooth
solutions are available, based on the adaptation of the theory of symmetric hyperbolic systems.
In \cite{LiYa1997}, a local existence theory was developed for the case that $c^ \alpha $ (with $0< \alpha \le 1$) is smooth across $\Gamma$,
using methods that are not applicable to the local existence theory
for the physical vacuum boundary. An existence theory for the small perturbation of a
planar wave was developed in \cite{XuYa2005}.  See also \cite{LiYa2000} and \cite{Yang2006}, for other features of the vacuum state problem.

The only existence theory for the physical vacuum boundary condition that we know of can be found in the recent paper by Jang and Masmoudi \cite{JaMa2009} for the 1-D compressible gas, wherein weighted Sobolev norms are used for the energy
estimates.  From these weighted norms, the regularity  of the solutions cannot be directly determined.  Letting $d$ denote the distance
function to the boundary $\partial I$, and letting $\| \cdot \|_0 $ denote
the $L^2(I)$-norm, an example of the type
of bound that is proved for the velocity field in \cite{JaMa2009} is the following:
\begin{align}
&\| d\, v\|_0^2 + \|d\, v_x\|_0^2 + \|d\, v_{xx} + 2 v_x\|_0^2  + \|d\, v_{xxx} + 2 v_{xx} - 2 d^{-1}\,  v_x\|_0^2 \n \\
&\qquad\qquad \qquad \qquad \qquad\qquad \qquad + \|d\, v_{xxxx} + 4 v_{xxx} - 4 d^{-1}\,  v_{xx}\|_0^2  < \infty  \,.
\label{jama}
\end{align} 
The problem with inferring the regularity of $v$ from this bound can already be seen at the level of an $H^1(I)$ estimate.   In particular,
the bound on the norm $\|d\, v_{xx} + 2 v_x\|_0^2$ only implies a bound on $\|d\, v_{xx}\|_0^2$ and $\|v_x\|_0^2$ if the integration by
parts on the cross-term,
$$
4\int_I  d\, v_{xx} \, v_x =  -2\int_I d_x \, |v_x|^2 \,,
$$
can be justified, which in turn requires having better regularity for $v_x$ than the a priori bounds provide.   Any methodology  which
seeks regularity in (unweighted) Sobolev spaces for solutions must contend with this type of issue.  We overcome this difficulty by 
constructing (sufficiently) smooth solutions to a degenerate parabolic regularization, and thus the sort of integration-by-parts difficulty
just described is overcome.   One can view our solutions as {\it degenerate viscosity solutions}.   The key to their construction is
our higher-order Hardy-type inequality that we provide below.

\section{Notation and Weighted Spaces}\label{notation}

\subsection{Differentiation and norms in the open interval $I$} 
 Throughout the paper the symbol $D$  will be used to 
 denote $\frac{\partial}{\partial x}$.

For integers $k\ge 0$,
we define the Sobolev space $H^k(I)$  to
be the completion of $C^\infty(I)$
in the norm
$$\|u\|_k := \left( \sum_{a\le k}\int_\Omega \left|   D^a u(x)
\right|^2 dx\right)^{1/2}\,.$$
For real numbers $s\ge 0$, the Sobolev spaces $H^s(I)$ and the norms $\| \cdot \|_s$ are defined by interpolation.
We use $H^1_0(I)$ to denote the subspace of $H^1(I)$ consisting of those functions $u(x)$ that vanish at $x=0$ and $x=1$.

\subsection{The embedding of a weighted Sobolev space} Using $d$ to denote the distance function to
the boundary $\Gamma$,  and letting $p=1$ or $2$,
the weighted Sobolev space  $H^1_{d^p}(I)$, with norm given by 
$\int_I d(x)^p (|F(x)|^2+| DF (x)|^2 )\, dx$ for any $F \in H^1_{d^p}(I)$, 
satisfies the following embedding:
$$H^1_{d^p}(I) \hookrightarrow   H^{1 - \frac{p}{2}}(I)\,,$$
so that there is a constant $C>0$ depending only on $I$ and $p$ such that
\begin{equation}\label{w-embed}
\|F\|_{1-p/2} ^2 \le C \int_I d(x)^p \bigl( |F(x)|^2 + \left| DF(x) \right|^2\bigr) \, dx\,.
\end{equation} 
See, for example,  Section 8.8 in Kufner \cite{Kufner1985}.

\section{A higher-order Hardy-type inequality}

We will make fundamental use of the following generalization of the well-known Hardy inequality to higher-order derivatives:

\begin{lemma}[Higher-order Hardy-type inequality]\label{Hardy}   Let $s\ge 1$ be a given integer, and suppose that
\begin{equation}\nonumber
u\in H^s(  I)\cap H^1_0(  I)\,. 
\end{equation} 
Then if $d$ denotes the distance fuction to $\partial I$, we have that $\d\frac{u}{d}\in H^{s-1}(  I)$ with 
\begin{equation}
\label{Hardys}
\d\left\|\frac{u}{d}\right\|_{s-1}\le C \|u\|_s.
\end{equation}
\end{lemma}
\begin{proof}
We use an induction argument.
The case $s=1$ is of course the classical Hardy inequality. Let us now assume that  the inequality (\ref{Hardys}) holds for a given $s\ge 1$, and suppose that
$$u\in H^{s+1}(  I)\cap H^1_0(  I)\,.$$ 
Using $D$ to denote $\frac{\partial}{\partial x}$, 
a simple computation shows that for $m \in \mathbb{N}  $,
\begin{equation}
\label{Hardy1}
D^m(\frac{u}{d})=\frac{f}{d^{m+1}},
\end{equation}
with $$f=\sum_{k=0}^m C_m^k D^{m-k}u\ k! (-1)^k d^{m-k}$$
for a constant $C_m^k$ depending on $k$ and $m$.
From the regularity of $u$, we see that $f\in H^1_0(  I)$. Next, with $D= \frac{\partial}{\partial x}$, we obtain the identity
\begin{align}
Df&=\sum_{k=0}^s C_s^k D^{s+1-k}u\ k! (-1)^k d^{s-k}+\sum_{k=0}^{s-1} C_s^k D^{s-k}u\ k! (-1)^k d^{s-k-1}(s-k)\n\\
&=D^{s+1}u\ s! (-1)^s d^s +\sum_{k=1}^s C_s^k D^{s+1-k}u\ k! (-1)^k d^{s-k}\n\\
& \qquad\qquad\qquad +\sum_{k=0}^{s-1} C_s^{k+1} D^{s-k}u\ (k+1)! (-1)^k d^{s-k-1}\n\\
&=D^{s+1}u\ s! (-1)^s d^s \,. \label{cs1}
\end{align}
Since $f\in H^1_0( I)$, we deduce from (\ref{cs1}) that for any $x\in (0,\frac{1}{2}]$, 
\begin{equation*}
f(x)=(-1)^s s!\ \int_0^x D^{s+1}u(y)\ y^s dy,
\end{equation*}
which by substitution in (\ref{Hardy1}) yields the identity
\begin{equation*}
\d D^s(\frac{u}{d})(x)=\frac{(-1)^s s!\ \int_0^x D^{s+1}u(y)\ y^s dy}{x^{s+1}},
\end{equation*}
which by a simple majoration provides the bound
\begin{equation*}
\d \bigl|D^s(\frac{u}{d})(x)\bigr|\le s!\ \frac{\psi_1(x)\int_0^x |D^{s+1}u(y)|\  dy}{d(x)},
\end{equation*}
where $\psi_1$ is the piecewise affine function equal to $1$ on $[0,\frac{1}{2}]$ and to $0$ on $[\frac{3}{4},1]$. Next, for any $x\in [\frac{1}{2},1)$, we obtain similarly that
\begin{equation*}
\d \bigl|D^s(\frac{u}{d})(x)\bigr|\le s!\ \frac{\psi_2(x)\int_x^1 |D^{s+1}u(y)|\  dy}{d(x)},
\end{equation*}
where $\psi_2$ is the piecewise affine function equal to $0$ on $[0,\frac{1}{4}]$ and to $1$ on $[\frac{1}{2},1]$.
so that for any $x\in I$:
\begin{equation}
\label{Hardy2}
\d \bigl|D^s(\frac{u}{d})(x)\bigr|\le s!\ \frac{\psi_1(x)\int_0^x |D^{s+1}u(y)|\  dy+\psi_2(x)\int_x^1 |D^{s+1}u(y)|\  dy}{d(x)}.
\end{equation}
Now, with $g=\psi_1(x)\int_0^x |D^{s+1}u(y)|\  dy+\psi_2(x)\int_x^1 |D^{s+1}u(y)|\  dy$, we notice that $g\in H_1^0( I)$, with $$\|g\|_1\le C\|D^{s+1}u\|_0.$$ Therefore, by the classical Hardy inequality, we infer from (\ref{Hardy2}) that
\begin{equation}
\label{Hardy3}
\bigl\|D^s(\frac{u}{d})\bigr\|_0\le C \|g\|_1\le C \|D^{s+1}u\|_0.
\end{equation}
Since we assumed in our induction process that our generalized Hardy inequality is true at  order $s$, we then have that
 $$\bigl\|\frac{u}{d}\bigr\|_{s-1}\le C \|u\|_s,$$ 
which, together with (\ref{Hardy3}), implies that
$$\bigl\|\frac{u}{d}\bigr\|_{s}\le  C \|u\|_{s+1},$$
and thus establishes the property at order $s+1$, and concludes the proof.
\end{proof}
 In order to obtain estimates independent of a regularization parameter $\kappa$ defined in Section \ref{statement},  we will also need the following Lemma, whose proof can be found in Lemma 1, Section 6 of \cite{CoSh2006}:
\begin{lemma}\label{kelliptic}
 Let $\kappa>0$ and $g\in L^\infty(0,T;H^s(I)))$ be given, and let $f\in H^1(0,T;H^s(I))$ be such that $$f+\kappa f_t=g\ \ \ \text{in}\ (0,T)\times I.$$ Then, $$\|f\|_{L^\infty(0,T;H^s(I))}\le C\, \max\{\|f(0)\|_s,\|g\|_{L^\infty(0,T;H^s(I))}\}.$$
\end{lemma}

\section{A degenerate parabolic approximation of the compressible Euler equations in vacuum}
\label{statement}

For $ \kappa >0$, we consider the following nonlinear degenerate parabolic approximation of  the compressible Euler
system (\ref{ce0}):
\begin{subequations}
  \label{approximate}
\begin{alignat}{2}
\rho_0  v_t + (\rho_0^2/\eta_x^2)_x&= \kappa [\rho_0^2 v_x]_x &&\text{in} \ \  I \times (0,T] 
\,, \label{approximate.a}\\
(\eta,v)&= (e,u_0 ) \ \ \ &&\text{on} \ \  I \times \{t=0\} \,, \label{approximate.b}\\
\rho_0& = 0 \ \ &&\text{ on }  \Gamma \,, \label{approximate.c}
\end{alignat}
\end{subequations}
with $ \rho _0(x) \ge C \operatorname{dist}( x, \Gamma ) $ for $x \in I$ near $\Gamma$.

We will first obtain the existence of a solution to (\ref{approximate}) on a short time interval $[0,T_ \kappa ]$ (with $T_ \kappa $ depending a priori on $\kappa$). We will then perform energy estimates on this solution that will show that the time of existence, in fact,  does not depend on $\kappa$, and moreover that our a priori estimates for this sequence of solutions  is also independently of 
$\kappa$. 
The existence of a solution to the compressible Euler equations (\ref{ce0}) is obtained as  the weak limit as $ \kappa \to 0$  of the sequence of solutions  to (\ref{approximate}).

\section{Solving the parabolic $\kappa$-problem (\ref{approximate}) by a fixed-point method}

For notational convenience, we will write 
$$ \eta ' = \frac{\partial \eta}{ \partial x}$$
and
similarly for other functions.  

\subsection{Assumption on initial data}  
Given $u_0$ and $\rho_0$,  and using the fact that $\eta(x,0)=x$, the quantity $v_t|_{t=0}$ for the degenerate parabolic $ \kappa $-problem is computed using (\ref{approximate.a}):
$$
v_t|_{{t=0}} = \left.\left(\frac{ \kappa }{\rho_0}  [\rho_0^2 v']'      - {\frac{1}{\rho_0}} (\rho_0^2/\eta'^2)'  \right)\right|_{t=0} = 
\left(\frac{ \kappa }{\rho_0}  [\rho_0^2 u_0']'      - 2 \rho_0' \right) \,.
$$
Similarly, for all $k \in \mathbb{N}  $,
$$
\partial_t^k v|_{{t=0}} = \left. \frac{\partial^k}{\partial t^k}\left(\frac{ \kappa }{\rho_0}  [\rho_0^2 v']'      - {\frac{1}{\rho_0}} (\rho_0^2/\eta'^2)'  \right)\right|_{t=0} \,.
$$
These formulae make it clear that  each $\partial_t^k v|_{t=0}$ is a function of space-derivates of $u_0$ and $\rho_0$.

For the purposes of constructing solutions to the degenerate parabolic $ \kappa $-problem (\ref{approximate}), it is convenient  to
smooth the initial data.   By standard convolution methods, we assume that $u_0 \in C^ \infty (I)$ and that $\rho_0 \in C^ \infty (I)$
and satisfies (\ref{degen1}) and (\ref{degen2}).    

\subsection{Functional framework for the fixed-point scheme and some notational conventions}

For $T>0$, we shall denote by $\mathcal{X}_T$ the following Hilbert space:
\begin{align}
\mathcal{X}_T=\{&v\in L^2(0,T;H^2(I))|\ \partial_t^4 v\in L^2(0,T;H^1(I)), \ \rho_0 \partial_t^4 v\in L^2(0,T;H^2(I)) \\
& \ \  \partial_t^3 v\in L^2(0,T;H^2(I)), \ \rho_0 \partial_t^3 v\in L^2(0,T;H^3(I)) \}\,,
\end{align}
endowed with its natural Hilbert norm:
\begin{align*}
\|v\|_{\mathcal{X}_T}^2 & = \|  v\|_{L^2(0,T;H^2(I))}^2   +  \|  \partial_t^4 v\| _{L^2(0,T;H^1(I))}^2 + \| \rho_0  \partial_t^4 v\| _{L^2(0,T;H^2(I))}^2 \\
& \qquad 
 + \|  \partial_t^3 v\| _{L^2(0,T;H^2(I))}^2 + \| \rho_0  \partial_t^3 v\| _{L^2(0,T;H^3(I))}^2  \,.
\end{align*} 

For $M>0$  given sufficiently large, we define the following closed, bounded, convex subset of $\mathcal{X}_T$:
\begin{align}\label{ctm}
\mathcal{C}_T(M)=\{ v\in \mathcal{X}_T \  : \ \|v\|_{\mathcal{X}_T}\le M\},
\end{align}
which is indeed non empty if $M$ is large enough.   Henceforth, we assume that $T>0$ is given such that independently
of the choice of $v \in \mathcal{C} _T(M)$,  
$$
\eta(x,t) =x + \int_0^t v(x,s)ds
$$
is injective for $t \in [0,T]$, and that ${\frac{1}{2}} \le \eta'(x,t) \le {\frac{3}{2}} $  for $t\in [0,T]$ and $x \in \overline I$.  This can
be achieved by taking $T>0$ sufficiently small:  with $e(x)=x$, notice that
$$
\| \eta'( \cdot ,t) -e\|_1 = \| \int_0^t v'(\cdot ,s) ds\|_1 \le \sqrt{T}M \,.
$$
The space $ \mathcal{X} _T$  will be appropriate for our fixed-point methodology to prove existence of a solution to our $\kappa$-regularized parabolic problem (\ref{approximate}).

Finally, we  define the polynomial function $ \mathcal{N} _0$  of norms of the  initial data as follows:
\begin{equation}\label{N0}
\mathcal{N} _0 = P( \| \sqrt{\rho_0} \partial_t^5 v'(0)\|_0,  \| \partial_t^4 v(0)\|_1,  \| \partial_t^3 v(0)\|_2 , \| \partial_t^2 v(0)\|_2 , 
\| \partial_t v(0)\|_2, \|  u_0\|_2) \,.
\end{equation} 

\subsection{A theorem for the existence and uniqueness of solutions to the parabolic $\kappa $-problem}
 We will make use of the Tychonoff fixed-point Theorem in our fixed-point procedure
(see, for example, \cite{Deimling1985}). Recall that this states that for a reflexive separable
Banach space $\mathcal{X}_T$, and $\mathcal{C}_T(M) \subset \mathcal{X}_T$ a closed, convex, bounded subset, if $F : \mathcal{C}_T(M) \to \mathcal{C}_T(M)$ is
weakly sequentially continuous into $\mathcal{X}_T$, then $F$ has at least one fixed point.

\begin{theorem}[Solutions to the parabolic $ \kappa $-problem]\label{thm_ksoln}
Given our smooth data, for $T $ taken sufficiently small,  there exists a unique solution $v \in \mathcal{X}_T$ to the degenerate
parabolic $ \kappa $-problem (\ref{approximate}).
\end{theorem}

\subsection{Linearizing the degenerate parabolic $ \kappa $-problem}
Given $\bar v \in \mathcal{C}_T(M)$, and defining $\bar \eta (x,t) =x + \int_0^t \bar v(x,\tau)d\tau$, we consider   the linear equation for $v$:
\begin{equation}
\label{linear1}
\rho_0 v_t+ \left[\frac{\rho_0^2}{\bar\eta'^2}\right]'-\kappa [\rho_0^2 v']'=0 \,.
\end{equation}
We will prove the following:
\begin{enumerate}
\item $v$ is a unique solution to (\ref{linear1});
\item $v \in \mathcal{C}_T(M)$ for $T$ taken sufficiently small;
\item the map $\bar v \mapsto v: \mathcal{C}_T(M) \to \mathcal{C}_T(M)$, and is sequentially weakly continuous in $\mathcal{X}_T$.
\end{enumerate}
The solution to our parabolic $ \kappa $-problem (\ref{approximate}) will then be obtained as a fixed-point of the map $\bar v \mapsto v$
in $\mathcal{X}_T$ via the Tychonoff fixed-point Theorem.

In order to use our higher-order  Hardy-type inequality, Lemma \ref{Hardy}, it will be convenient to introduce the new variable 
$$X=\rho_0 v'$$
 which then belongs to $H_0^1( I)$. By a simple computation, we see that (\ref{linear1}) is equivalent to
\begin{equation*}
v_t'+ \left[\frac{2}{\bar\eta'}\Bigl(\frac{\rho_0}{\bar\eta'}\Bigr)'\right]'-\kappa \left[\frac{1}{\rho_0}\Bigl[\rho_0^2 v'\Bigr]'\right]'=0,
\end{equation*}
and hence that
\begin{subequations}
\label{div}
\begin{alignat}{2}
\frac{X_t}{\rho_0}-\kappa \left[\frac{1}{\rho_0} (\rho_0\ X)'\right]'&=-\left[\frac{2}{\bar\eta'}\Bigl(\frac{\rho_0}{\bar\eta'}\Bigr)'\right]' && \  \text{ in }\ [0,T]\times  I,\\
X&=0 && \ \text{ on }\ \ [0,T]\times\partial I,\\
X|_{t=0}&=\rho_0 u_0' && \ \text{ on } \  I.
\end{alignat}
\end{subequations}
We shall therefore solve the degenerate linear parabolic problem (\ref{div}) with Dirichlet boundary conditions, which (as we will prove) will surprising admit a solution with arbitrarily high space regularity (depending on the regularity of the right-hand side and the initial data, of course), and not just an $H_0^1(I)$-type  weak solution. After we obtain the solution  $X$, we will then easily find our solution $v$ to (\ref{linear1}).

In order to construct our fixed-point, we will need to obtain estimates for $X$ (and hence $v$)
with high space regularity; in particular,  we will need to study the   fifth  time-differentiated problem.  For this purpose, it is convenient to define
the  new variable 
\begin{equation}\label{defineY}
Y=\partial_t^5 X =\rho_0 \partial_t^5 v'\,,
\end{equation} 
and consider the following equation for $Y$:
\begin{subequations}
\label{divt}
\begin{alignat}{2}
\frac{Y_{t}}{\rho_0}-\kappa \left[\frac{1}{\rho_0} (\rho_0\ Y)'\right]'&=-\partial_t^5\left[\frac{2}{\bar\eta'}\Bigl(\frac{\rho_0}{\bar\eta'}\Bigr)'\right]' && \ \hbox{in}\ [0,T]\times  I, \label{divt.a}\\
Y&=0&& \ \hbox{on}\ [0,T]\times\partial I, \label{divt.b} \\
Y|_{t=0}&=Y_{\text{init}} && \ \hbox{in}\  I \,,  \label{divt.c}
\end{alignat}
\end{subequations}
where $Y_{\text{init}} = \rho_0 \partial_t^5 v'|_{t=0}$.

\subsection{Existence of a weak solution to the linear problem (\ref{divt}) by a Galerkin scheme}

Let $(e_n)_{n\in\N}$ denote a Hilbert basis of $H_0^1( I)$, with each $e_n$ being of class $H^2( I)$. Such a choice of basis is indeed possible as we can take for instance the eigenfunctions of the Laplace operator on $I$ with vanishing Dirichlet boundary conditions. We then define the Galerkin approximation at order $n\ge 1$ of (\ref{div}) as being under the form $Y_n=\sum_{i=0}^n \lambda_i^n(t) e_i$ such that:
\begin{subequations}
\label{divn}
\begin{align}
\forall k\in\{0,...,k\},\ \left(\frac{{Y_n}_t}{\rho_0}  + \kappa  (\rho_0 Y_n)' \,,  \frac{ e_k' }{\rho_0} \right)_{L^2( I)} &=\left(\partial_t^5\bigr[\frac{2}{\bar\eta'}\bigl[\frac{\rho_0}{\bar\eta'}\bigr]'\bigl] \ ,\ e_k'\right)_{L^2( I)}\ \hbox{in}\ [0,T],\\
\lambda_i^n(0)&=(Y_{\text{init}},e_i)_{L^2( I)}.
\end{align}
\end{subequations}
Since each $e_i$ is in $H^2( I)\cap H_0^1( I)$, we have by our high-order Hardy-type inequality (\ref{Hardy}) that 
$$\frac{e_i}{\rho_0}\in H^1( I) \,;$$
therefore, each integral written in (\ref{divn}) is well-defined. 
Furthermore, as the $e_i$ are linearly independent, so are the $\frac{e_i}{\sqrt{\rho_0}}$ and therefore the determinant of the matrix 
$$\Bigl(\bigl(\frac{e_i}{\sqrt{\rho_0}},\frac{e_j}{\sqrt{\rho_0}}\bigr)_{L^2( I)}\Bigr)_{(i,j)\in\N_n=\{1,...,n\}}$$ 
is nonzero. This implies that our finite-dimensional Galerkin approximation (\ref{divn}) is a well-defined first-order differential system of order $n+1$, which  therefore has a solution on a time interval $[0,T_n]$, where $T_n$ a priori depends on the dimension $n$ of
the Galerkin approximation. 
In order to prove that $T_n=T$, with $T$ independent of $n$, we notice that since $Y_n$ is a linear combination of the $e_i$ ($i\in \N_n$), we have that
\begin{equation}
\label{g1}
\left(\frac{{Y_n}_t}{\rho_0}-\kappa [\frac{1}{\rho_0} (\rho_0\ Y_n)']',\ Y_n\right)_{L^2( I)} =\left(\partial_t^5\bigr[\frac{2}{\bar\eta'}\bigl[\frac{\rho_0}{\bar\eta'}\bigr]'\bigl]\ ,\ Y_n'\right)_{L^2( I)}.
\end{equation}
Since  ${Y_n}\in H_0^1( I)$ and $[\frac{1}{\rho_0} (\rho_0\ Y_n)']'\in H^1( I)$, integration by parts yields 
\begin{align}
- \int_I [\frac{1}{\rho_0} (\rho_0\ Y_n)']'\ Y_n=\int_I [\frac{1}{\rho_0} (\rho_0\ Y_n)']Y_n'
 =\int_I {Y_n'}^2+ \int_I \rho_0'\frac{Y_n}{\rho_0} Y_n'.
 \label{g2}
\end{align}
Next, using our higher-order Hardy-type inequality, we see that $\frac{Y_n}{\rho_0}\in H^1( I)$, and thus
$$
\int_I \rho_0'\frac{Y_n}{\rho_0} Y_n' =-\int_I \rho_0'\frac{1}{\rho_0}  Y_n' Y_n+\int_I \frac{\rho_0'^2}{\rho_0^2} {Y_n^2}-\int_I \rho_0''\frac{Y_n^2}{\rho_0},
 $$
which implies that
$$
\int_I \rho_0'\frac{Y_n}{\rho_0} Y_n' =\frac{1}{2} \int_I \frac{\rho_0'^2}{\rho_0^2} {Y_n^2}-\frac{1}{2} \int_I \rho_0''\frac{Y_n^2}{\rho_0}.
 $$
Substitution of  this identity into (\ref{g1}) and (\ref{g2}) yields
\begin{equation*}
\frac{1}{2} \bigl[\frac{d}{dt}\int_I\frac{Y_n^2}{\rho_0}- \kappa \int_I \rho_0''\frac{Y_n^2}{\rho_0}\bigr] +\kappa\int_I Y_n'^2+\frac{1}{2}\kappa\int_I \frac{\rho_0'^2}{\rho_0^2} Y_n^2=
 -\int_I \partial_t^5\bigl[\frac{2}{\bar\eta'}\bigl[\frac{\rho_0}{\bar\eta'}\bigr]'\bigr] Y_n'\ ,
\end{equation*}
which shows that (since our given $\bar v\in \mathcal{C}_T(M)$):
\begin{equation*}
 \frac{d}{dt}\int_I\frac{Y_n^2}{\rho_0}- \kappa \|\rho_0''\|_{L^\infty} \int_I \frac{Y_n^2}{\rho_0}+{\kappa}\int_I Y_n'^2+\kappa\int_I \rho_0'^2\frac{Y_n^2}{\rho_0^2}\le C_\kappa,
\end{equation*}
for a constant $C_ \kappa $ depending on $ \kappa $.
Consequently, $T_n=T$ with $T$ independent of $n$, and
\begin{equation}
 \sup_{[0,T]} \int_I \frac{Y_n^2}{\rho_0} +{\kappa}\int_0^T\int_I Y_n'^2\le C_\kappa T + C \mathcal{N} _0 \,,
\end{equation}
where $ \mathcal{N} _0$ is defined in (\ref{N0}).
Thus, there exists a  subsequence of $(Y_n)$ which converges weakly to some $Y\in L^2(0,T;H_0^1( I))$, which satisfies
\begin{equation}
\label{g3}
 \sup_{[0,T]}\int_I \frac{Y^2}{\rho_0} +{\kappa}\int_0^T\int_I Y'^2\le C_\kappa T + C \mathcal{N} _0 \,.
\end{equation}
With (\ref{defineY}), we see that
\begin{equation}\label{pt5v}
\|\rho_0 \partial_t^5v' \|_{L^2(0,T; H^1(I))} \le C_\kappa T + C \mathcal{N} _0 \,.
\end{equation} 
Furthermore, it can also be shown using standard arguments that $Y$ is a solution of (\ref{divt}) (where (\ref{divt.a}) is satisfied almost everywhere in $[0,T]\times I$ and holds in a variational sense for all test functions in $L^2(0,T;H_0^1( I))$), and that
\begin{equation}\nonumber
\frac{Y_t}{\rho_0} \in L^2(0,T; H ^{-1} (I)) \,.
\end{equation} 

Now,  with the functions
$$
X_i = \rho_0\left.\frac{ \partial^i v'}{ \partial t^i}   \right|_{t=0} \text{ for }  i=0,1,2,3,4\,,
$$
we define
\begin{equation}
\label{Z}
Z(t,x)=  \int_0^t Y(\cdot,x) =   \rho_0(x) \int_0^t  \partial_t^5 v' ( \cdot, x ) =  \rho_0(x) \partial_t^4 v'( t, x) - X_4(x)  \,,
\end{equation}
and
\begin{equation}
\label{X}
X(t,x)=\sum_{i=0}^4 X_i t^i+\int_0^t\int_0^{t_4}\int_0^{t_3}\int_0^{t_2} Z(t_1,x) dt_1dt_2dt_3dt_4 \,.
\end{equation}

We then see that $X\in C^0([0,T];H_0^1( I))$ is a solution of (\ref{div}), with $\partial_t^5X=Y$. In order to obtain a fixed-point for the map
$\bar v \mapsto v$, we need to establish better space regularity for $Z$, and hence $X$ and $v$.

\subsection{Improved space regularity for $Z$}
We introduce the variable $\check v$ defined by
\begin{equation*}
\check v(t,x)=\int_0^x \frac{X(t,\cdot)}{\rho_0(\cdot)} \,,
\end{equation*}
so that $\check v$ vanishes at $x=0$, and will us to employ the Poincar\'{e} inequality with this variable.

It is easy to see that
\begin{equation*}
X=\rho_0 \check v',
\end{equation*}
and
\begin{equation}\label{Z2}
Z=\rho_0 \partial_t^4\check v'.
\end{equation}
Thanks to the standard Hardy inequality, we thus have that
\begin{equation*}
\|\partial_t^4\check v'\|_0\le C \|Z\|_1,
\end{equation*}
and hence by Poincar\'e's inequality,
\begin{equation}
\label{g4}
\|\partial_t^4\check v(t,\cdot)\|_1\le C\|Z(t,\cdot)\|_1.
\end{equation}

With 
$$
F_0= \frac{Y_{\text{init}}}{\rho_0} + \left. \partial_t^4\bigr[\frac{2}{\bar\eta'}\bigl[\frac{\rho_0}{\bar\eta'}\bigr]'\bigl]' \right|_{t=0}\,,
$$
our starting point is the equation
\begin{equation*}
\frac{Y}{\rho_0}-\kappa [\frac{1}{\rho_0} (\rho_0\ Z)']' =-\partial_t^4\bigr[\frac{2}{\bar\eta'}\bigl[\frac{\rho_0}{\bar\eta'}\bigr]'\bigl]' +
F_0
\ \hbox{in}\ [0,T]\times  I,
\end{equation*}
which follows from our definition of $Z$ given in (\ref{Z}) and time-integration of (\ref{divt.a}).
From this equation,  we infer that
\begin{equation*}
\kappa \bigl\|[\frac{1}{\rho_0}(\rho_0\ Z)']'\bigr\|_0\le \bigl\|\partial_t^4\bigl[\frac{2}{\bar\eta'}\bigl[\frac{\rho_0}{\bar\eta'}\bigr]'\bigl]'\bigr\|_0
+\bigl\|\frac{Y}{\rho_0}\bigr\|_0 + \|F_0\|_0 \,.
\end{equation*}
By the standard Hardy inequality and the fact that $\bar v\in \mathcal{C}_T(M)$, we obtain the estimate
\begin{equation*}
\kappa \bigl\|[\frac{1}{\rho_0}(\rho_0\ Z)']'\bigr\|_0\le C_M \sqrt{T} +
C\|{Y}\|_1 +  \mathcal{N} _0 \,,
\end{equation*}
where $C_M$ is a constant that depends on $M$.  In particular, using (\ref{Z2}), we see that
$$
\frac{1}{\rho_0}(\rho_0\ Z)' = \rho_0 \partial_t^4\check v''+2 \rho_0'\partial_t^4\check v'
$$
so that
\begin{equation*} 
%
\kappa \bigl\|\rho_0 \partial_t^4\check v'''+3 \rho_0'\partial_t^4\check v'' + 2 \rho_0'' \partial_t^4 \check v'\bigr\|_0
\le C_M \sqrt{T} + C\|{Y}\|_1 +  \mathcal{N} _0 \,,
\end{equation*}
which implies that
\begin{subequations}
\begin{align}
\kappa \bigl\|(\rho_0 \partial_t^4\check v)'''\bigr\|_0 & \le C_M \sqrt{T} +
C\|{Y}\|_1 +  \mathcal{N} _0+ \kappa \|\rho_0''' \partial_t^4 \check v\|_0  + \kappa \|3\rho_0''\partial_t^4 \check v'\|_0\n\\
& \le C_M \sqrt{T} +
C\|{Y}\|_1+  \mathcal{N} _0+ \kappa( \|\rho_0'''\|_{L^\infty}+  3 \|\rho_0''\|_{L^\infty} )\|Z\|_1\n \,,
\end{align}
\end{subequations}
where we have used (\ref{g4}) in the second inequality above.     Having established in (\ref{g4})  that $(\rho_0\partial_t^4\check v)\in H_0^1( I)$, elliptic regularity shows that
\begin{equation}
\label{g5}
\kappa \|\rho_0 \partial_t^4\check v\|_3 \le  C_M \sqrt{T} +
C\|{Y}\|_1 + \mathcal{N} _0 + \kappa( \|\rho_0'''\|_{L^\infty}+  3 \|\rho_0''\|_{L^\infty} )\|Z\|_1\,.
\end{equation}
Now, thanks to our high-order Hardy-type inequality, we infer from from (\ref{g5}) that
\begin{equation}
\label{g6}
\kappa \| \partial_t^4\check v\|_2 \le  C_M \sqrt{T} +
C\|Y\|_1 +  \mathcal{N} _0  + \kappa( \|\rho_0'''\|_{L^\infty}+  3 \|\rho_0''\|_{L^\infty} )\|Z\|_1\,.
\end{equation}
Next we see that (\ref{g5}) implies that
\begin{equation*}
\kappa \|\rho_0 \partial_t^4\check v'+\rho' \partial_t^4\check v\|_2  \le  C_M \sqrt{T} +
C\|Y\|_1+ \mathcal{N} _0 + \kappa( \|\rho_0'''\|_{L^\infty}+  3 \|\rho_0''\|_{L^\infty} )\|Z\|_1\,,
\end{equation*}
which thanks to (\ref{g6}) and (\ref{Z2}) implies that
\begin{equation}
\label{g7}
\kappa \|Z\|_2 \le  C_M \sqrt{T} +
C\|Y\|_1 +  \mathcal{N} _0+  \kappa( \|\rho_0'''\|_{L^\infty}+  3 \|\rho_0''\|_{L^\infty} )\|Z\|_1\,.
\end{equation}

\subsection{Definition of $v$ and the existence of a fixed-point}
We are now in a position to define $v$ in the following fashion: let us first define on $[0,T]$
\begin{equation*}
f(t)=u_0(0)-\int_0^t \frac{1}{\rho_0}\bigl[\frac{\rho_0^2}{\bar\eta'^2}\bigr]'(\cdot,0)+\kappa\int_0^t\frac{1}{\rho_0} [\rho_0 X]'(\cdot,0),
\end{equation*}
which is well-defined thanks to (\ref{g7}) and (\ref{g3}). We next define
\begin{equation}
v(t,x)=f(t)+\check v(t,x).
\end{equation}
We then notice that from (\ref{div}), we immediately have that
\begin{equation*}
v_t'+ \bigl[\frac{1}{\rho_0} \bigl[\frac{\rho_0^2}{\bar\eta'^2}\bigr]'\bigr]'-\kappa \bigl[\frac{1}{\rho_0} [\rho_0^2 v']'\bigr]'=0,
\end{equation*}
from which we infer that in $[0,T]\times I$
\begin{equation*}
v_t+ \frac{1}{\rho_0} \bigl[\frac{\rho_0^2}{\bar\eta'^2}\bigr]'-\kappa \frac{1}{\rho_0} [\rho_0^2 v']'=g(t),
\end{equation*}
for some function $g$ depending only on $t$. By taking the trace of this equation on the left end-point $x=0$, we see that
\begin{equation*}
v_t(t,0)+ \frac{1}{\rho_0} \bigl[\frac{\rho_0^2}{\bar\eta'^2}\bigr]'(t,0)-\kappa \frac{1}{\rho_0} [\rho_0^2 v']'(t,0)=g(t),
\end{equation*}
which together with the identity
$$v_t(t,0)=f_t(t)=-\frac{1}{\rho_0} \bigl[\frac{\rho_0^2}{\bar\eta'^2}\bigr]'(t,0)+\kappa \frac{1}{\rho_0} [\rho_0^2 v']'(t,0)$$ 
shows that $g(t)=0$. Therefore, $v$ is a solution of (\ref{linear1}), and also satisfies by construction $v(0,\cdot)=u_0(\cdot)$.

We can now establish the existence of a fixed-point for  the mapping $\bar v\rightarrow v$ in $\mathcal{C}_T(M)$, with $T$ taken sufficiently small and depending a priori on $\kappa$. 
We first notice that, thanks to the estimates (\ref{g7}) and (\ref{g3}), we have the inequality
\begin{equation*}
\|\partial_t^4 f(t)\|_{L^2(0,T)}\le \mathcal{N} _0+\sqrt{T} C_M,
\end{equation*}
which together with (\ref{g6}) and (\ref{g3}) provides the estimate
\begin{equation}
\label{g8}
\|\partial_t^4 v\|_{L^2(0,T;H^2( I))}\le \mathcal{N} _0+C_\kappa \sqrt{T} C_M \,.
\end{equation}
Then,  (\ref{g5}) implies that
\begin{equation}
\label{g8b}
\|\rho_0\partial_t^4 v\|_{L^2(0,T;H^3( I))}+
\|\partial_t^4 v\|_{L^2(0,T;H^2( I))}\le \mathcal{N} _0+C_\kappa \sqrt{T} C_M \,,
\end{equation}
and combining this with (\ref{pt5v}) shows that
\begin{equation}
\label{g8c}
\|\partial_t^5 v\|_{L^2(0,T;H^1( I))}+
\|\rho_0 \partial_t^5 v'\|_{L^2(0,T;H^1( I))}\le \mathcal{N} _0+C_\kappa \sqrt{T} C_M \,,
\end{equation}

In turn,  (\ref{g8b}) shows that for 
\begin{equation}
\label{g9}
T\le \frac{\mathcal{N} _0^2}{C_\kappa C_M},
\end{equation}
 $v\in \mathcal{C}_T(M)$.
Moreover, it is clear that there is only one solution $v\in L^2(0,T;H^2( I))$  of (\ref{linear1}) with $v(0)=u_0$ (where this initial
condition is well-defined as $\|v_t\|_{L^2(0,T;H^1( I))}\le \mathcal{N} _0 +C_\kappa \sqrt{T} C_M$), 
since if we denote by $w$ another solution with the same regularity, the difference $\delta v=v-w$ satisfies $\delta v(0,\cdot)=0$ with $\rho_0\delta v_t-\kappa [\rho_0^2\delta v']'=0$ which implies
$$\frac{1}{2}\frac{d}{dt}\int_I \rho_0 \delta v^2 +\kappa \int_I \rho_0^2 \delta v^2=0,$$
which with $\delta v(0,\cdot)=0$ implies $\delta v=0$. So the mapping $\bar v\rightarrow v$ is well defined, and thanks to (\ref{g8}) is a mapping from $\mathcal{C}_T(M)$ into itself for $T=T_\kappa$ satisfying \ref{g9}). As it is furthermore clear that it is weakly continuous in the $L^2(0,T_\kappa;H^2( I))$ norm, the Tychonoff fixed-point theorem \cite{Deimling1985} provides us with the existence of a fixed-point to this mapping. Such a fixed-point, which we denote by $v_\kappa$, is a solution of the nonlinear degenerate parabolic 
$\kappa$-problem (\ref{approximate}), with  initial condition $v_\kappa(0,\cdot)=u_0(\cdot)$.  It should be clear that the fixed-point
$v_ \kappa $ also satisfies (\ref{pt5v}) and (\ref{g8b}) so that
\begin{align}
&\|\rho_0 \partial_t^4 v_ \kappa \|_{L^2(0,T;H^3( I))} +
\|\partial_t^4 v_ \kappa \|_{L^2(0,T;H^2( I))} \n \\
& \qquad +
\|\rho_0 \partial_t^5 v'_ \kappa \|_{L^2(0,T;H^1( I))} +
\|\partial_t^5 v_ \kappa \|_{L^2(0,T;H^1( I))}
\le M \,,
\label{g8d}
\end{align}

In the next section, we establish $\kappa $-independent estimates for $v_\kappa$ in $L^2(0,T_\kappa;H^2( I))$ (which are indeed
possible because our parabolic approximate $ \kappa $-problem respects the stucture of the original compressible Euler
equations (\ref{ce0})), from
which we infer a short time-interval of existence $[0,T]$, with $T$ independent of $\kappa$.   These $ \kappa $-independent estimates will
allows us to pass to the weak limit of the sequence $v_ \kappa $ as $ \kappa \to 0$ to obtain  the solution to (\ref{ce0}).

\section{Asymptotic estimates for  $v_\kappa$ which are independent of $\kappa$.} 

\subsection{The higher-order energy function appropriate for the asymptotic estimates as $ \kappa \to 0$.}

Our objective in this section is to show that the higher-order energy function $E$ defined in (\ref{E}) satisfies the inequality
\begin{equation}\label{poly}
\sup_{t \in [0,T]} E(t) \le M_0 + C\,T\, P( \sup_{t \in [0,T]} E(t)) \,,
\end{equation} 
where $P$ denotes a polynomial function, and
for $T>0$ taken sufficiently small,  with $M_0$ defined in (\ref{M0}).      The norms in $E$ are for solutions $v_ \kappa $ to
our degenerate parabolic $ \kappa $-problem (\ref{approximate}).   According to Theorem \ref{thm_ksoln}, 
$v_\kappa \in X_{T_ \kappa }$ with the additional bound $ \|\partial_t^4 v_ \kappa \|_{L^2(0,T_ \kappa ;H^2( I))} < \infty $ provided by (\ref{g8}).
As such,  the energy function  $E$ is continuous with respect to $t$, and the inequality (\ref{poly}) would thus establish a time
interval of existence and bound which are both independent of $ \kappa $.
For the sake of notational convenience,  we shall denote $v_\kappa$ by $\tilde v$.

\subsection{A $\kappa$-independent energy estimate on the fifth time-differentiated problem}
Our starting point shall be the fifth time-differentiated problem of (\ref{approximate}) for which we have, 
by naturally using $\partial_t^5 \v\in L^2(0,T_\kappa;H^1( I))$ as a test function, the following identity:
\begin{equation}
\label{a1}
\underbrace{\frac{1}{2} \frac{d}{dt}\int_I{\rho_0}|\partial_t^5\v|^2}_{\mathcal{I} _1} \
- \ \underbrace{\int_I \partial_t^5\bigl[\frac{\rho_0^2}{\tilde\eta'^2}\bigr] \partial_t^5\v'}_{ \mathcal{I} _2}  \
+ \ \underbrace{\kappa\int_I \rho_0^2(\partial_t^5\v')^2=0 }_{ \mathcal{I} _3}\,.
\end{equation}
In order to form the exact time-derivative in term $ \mathcal{I} _1$, we rely on the fact that  solutions we constructed to
(\ref{approximate}) satisfy
$ \partial_t^6 v \in L^2(0,T_ \kappa ; L^2 (I))$, which follows from the relation
\begin{equation*}
\partial_t^6 \v = \partial_t^5\left[\frac{2}{\tilde\eta'}\Bigl(\frac{\rho_0}{\tilde\eta'}\Bigr)'\right]+ \frac{\kappa}{\rho_0}\Bigl[\rho_0^2  \partial_t^5 v'\Bigr]' \,,
\end{equation*}
and the estimate (\ref{g8d}).
Upon integration in time, both the terms $ \mathcal{I} _1$ and $ \mathcal{I} _3$ provide  sign-definite energy contributions, so we
focus our attention on the nonlinear estimates required of the term $ \mathcal{I} _2$.

We see that
\begin{align*}
- \mathcal{I} _2&=2\int_I \partial_t^4\v'\bigl[\frac{\rho_0^2}{\tilde\eta'^3}\bigr] \partial_t^5\v'+\sum_{a=1}^4 c_a \int_I \partial_t^{5-a}\frac{1}{\tilde\eta'}\partial_t^{a}\frac{1}{\tilde\eta'}\rho_0^2\partial_t^5\v'\n\\
&=\frac{d}{dt}\int_I (\partial_t^4\v')^2\frac{\rho_0^2}{\tilde\eta'^3}+3\int_I (\partial_t^4\v')^2\v'\frac{\rho_0^2}{\tilde\eta'^4}+\sum_{a=1}^4 c_a \int_I  \partial_t^{5-a}\frac{1}{\tilde\eta'}\partial_t^{a}\frac{1}{\tilde\eta'}\rho_0^2\partial_t^5\v' \,.
\end{align*}
Hence integrating (\ref{a1})  from $0$ to $t\in[0,T_\kappa]$, we find that
\begin{align}
\label{a2}
\frac{1}{2} \int_I{\rho_0}\partial_t^5\v^2 (t) &+ \int_I (\partial_t^4\v')^2\frac{\rho_0^2}{\tilde\eta'^3} (t) +\kappa\int_0^t\int_I \rho^2(\partial_t^5\v')^2\n\\
&=\frac{1}{2} \int_I{\rho_0}\partial_t^5\v^2 (0) + \int_I (\partial_t^4\v')^2\frac{\rho_0^2}{\tilde\eta'^3} (0)-3\int_0^t\int_I (\partial_t^4\v')^2\v'\frac{\rho_0^2}{\tilde\eta'^4}\n\\
&\ \ \ -\sum_{a=1}^4 c_a \int_0^t\int_I  \partial_t^{5-a}\frac{1}{\tilde\eta'}\partial_t^{a}\frac{1}{\tilde\eta'}\rho_0^2\partial_t^5\v'\ .
\end{align}
We next show that all of the error terms, comprising the right-hand side of (\ref{a2}) can be bounded by $C t P( \sup_{[0,t]} E)$.
For the first spacetime integral appearing on the right-hand side of (\ref{a2}), it is clear that
\begin{equation}
\label{a3}
-3\int_0^t\int_I (\partial_t^4\v')^2\v'\frac{\rho_0^2}{\tilde\eta'^4} \le C t \ P(\sup_{[0,t]} E).
\end{equation}
We now study the last integrals on the right-hand side of (\ref{a2}). Using integration-by-parts in time, we have that
\begin{equation}
\label{a4}
\int_0^t\int_I  \partial_t^{5-a}\frac{1}{\tilde\eta'}\partial_t^{a}\frac{1}{\tilde\eta'}\rho_0^2\partial_t^5\v'=-\int_0^t\int_I  [\partial_t^{5-a}\frac{1}{\tilde\eta'}\partial_t^{a}\frac{1}{\tilde\eta'}]_t\rho_0^2\partial_t^4\v'+\left. \int_I  \partial_t^{5-a}\frac{1}{\tilde\eta'}\partial_t^{a}\frac{1}{\tilde\eta'}\rho_0^2\partial_t^4\v'\right|_0^t.
\end{equation}
We first consider the spacetime integral on the right-hand side of (\ref{a4}). As the sum is taken for $a=1$ to $4$, we then see that it will be written under the form of the sum of spacetime integrals of the following type:
\begin{align*}
I_1&=\int_0^t\int_I \rho_0 \partial_t^4\v' R(\tilde \eta) \rho_0\partial_t^4\v',\\
I_2&=\int_0^t\int_I \rho_0 \partial_t^3\v' v' \partial_t \v' R(\tilde \eta) \rho_0\partial_t^4\v',\\
I_3&=\int_0^t\int_I \rho_0 \partial_t^2\v' v' \partial_t^2 \v' R(\tilde \eta) \rho_0\partial_t^4\v',\\
I_4&=\int_0^t\int_I \rho_0 \partial_t^2\v' v' R(\tilde \eta) (\partial_t\v')^2\rho_0\partial_t^4\v'\ ,
\end{align*}
Where $R(\tilde\eta)$ denotes a rational function of $\tilde\eta$. We first immediately see that
\begin{equation}
\label{a5}
|I_1|\le C t \ P(\sup_{[0,t]} E).
\end{equation}
Next, we have that
\begin{align}
\label{a6}
|I_2| &\le C \int_0^t\int_I \|\rho_0 \partial_t^3\v'\|_{L^4} \|v'\|_{L^\infty} \|\partial_t \v'\|_{L^4} \|R(\tilde \eta)\|_{L^\infty} \|\rho_0\partial_t^4\v'\|_0\n\\
 &\le C \int_0^t\int_I \|\rho_0 \partial_t^3\v'\|_{H^{\frac{1}{2}}} \|v'\|_{L^\infty} \|\partial_t \v'\|_{H^{\frac{1}{2}}} \|R(\tilde \eta)\|_{L^\infty} \|\rho_0\partial_t^4\v'\|_0\n\\
& \le C t \ P(\sup_{[0,t]} E).
\end{align}
Similarly,
\begin{align}
\label{a6b}
|I_3| &\le C \int_0^t \|\rho_0 \partial_t^2\v'\|_{L^\infty} \|v'\|_{L^\infty} \|\partial_t^2 \v'\|_0\| R(\tilde \eta)\|_{L^\infty} \|\rho_0\partial_t^4\v'\|_0\n\\
 &\le C t \ P(\sup_{[0,t]} E),
\end{align}
and
\begin{align}
\label{a7}
|I_4| &\le C \int_0^t \|\rho_0 \partial_t^2\v'\|_{L^\infty} \|v'\|_{L^\infty} \|\partial_t \v'\|^2_{L^4}\| R(\tilde \eta)\|_{L^\infty} \|\rho_0\partial_t^4\v'\|_0\n\\
& \le C t \ P(\sup_{[0,t]} E),
\end{align}
where we have used the fact that in 1-D, $\|\cdot\|_{L^\infty}\le C \|\cdot\|_{H^1}$ and $\|\cdot\|_{L^4}\le C \|\cdot\|_{H^{\frac{1}{2}}}$.
Therefore, estimates (\ref{a2})--(\ref{a6}) provide us with
\begin{equation*}
\frac{1}{2} \int_I{\rho_0}\partial_t^5\v^2 (t) + \int_I (\partial_t^4\v')^2\frac{\rho_0^2}{\tilde\eta'^3} (t) +\kappa\int_0^t\int_I \rho_0^2(\partial_t^5\v')^2
\le M_0+ C t \ P(\sup_{[0,t]} E),
\end{equation*}
and thus, employing the fundamental theorem of calculus, 
\begin{align}
\label{a8}
&\frac{1}{2} \int_I{\rho_0}\partial_t^5\v^2 (t) + \int_I (\rho_0 \partial_t^4\v')^2 (t) +\kappa\int_0^t\int_I \rho_0^2(\partial_t^5\v')^2\n\\
&\qquad\qquad\qquad \qquad \le M_0+ C t \ P(\sup_{[0,t]} E)+ 3\int_I (\partial_t^4\v')^2(t){\rho_0^2}\int_0^t \frac{\v'}{\tilde\eta'^4}\n\\
 &\qquad\qquad\qquad \qquad \le M_0+ C t \ P(\sup_{[0,t]} E).
\end{align}

\subsection{Elliptic and Hardy-type estimates for $\partial_t^2 v(t)$}  Having obtained the energy estimate (\ref{a8}) for the
fifth time-differentiated problem, we can begin our bootstrapping
argument.  We now consider the third time-differentiated version of
 (\ref{approximate.a}),
\begin{equation*}
 \bigl[\partial_t^3\frac{\rho_0^2}{\tilde\eta'^2}\bigr]'-\kappa [\rho_0^2 \partial_t^3 \v']'=-\rho_0 \partial_t^4 \v,
\end{equation*}
which can be written as 
\begin{equation*}
 -2\bigl[\frac{\rho_0^2\partial_t^2 \v'}{\tilde\eta'^3}\bigr]'-\kappa [\rho_0^2 \partial_t^3 \v']'=-\rho_0\partial_t^4 \v+c_1\bigl[\frac{\rho_0^2\partial_t \v'\ \v'}{\tilde\eta'^4}\bigr]' +c_2 \bigl[\frac{\rho_0^2 \v'^3}{\tilde\eta'^5}\bigr]',
\end{equation*}
and finally rewritten as the following identity:
\begin{align}
 -2\bigl[{\rho_0^2\partial_t^2 \v'}\bigr]'-\kappa [\rho_0^2 \partial_t^3 \v']'=&-\rho_0\partial_t^4 \v+c_1\bigl[\frac{\rho_0^2\partial_t \v'\ \v'}{\tilde\eta'^4}\bigr]' +c_2 \bigl[\frac{\rho_0^2 \v'^3}{\tilde\eta'^5}\bigr]' \n\\
 &-2\bigl[{\rho_0^2\partial_t^2 \v'}\bigr]'(1-\frac{1}{\tilde\eta'^3})-6\rho_0^2\partial_t^2\v'\frac{\tilde\eta''}{\tilde\eta'^4}\ . \label{cs3}
\end{align}
Here, $c_1$ and $c_2$ are constants whose exact value is not important.

Therefore, using Lemma \ref{kelliptic} together with the fundamental theorem of calculus for the fourth term on the right-hand side
of (\ref{cs3}),  we obtain  that for any $t\in [0,T_\kappa]$,
\begin{align}
\label{a9}
\sup_{[0,t]} \bigl\|\frac{2}{\rho_0}\bigl[{\rho_0^2\partial_t^2 \v'}\bigr]'\bigr\|_0 \le & \sup_{[0,t]}\|\partial_t^4 \v\|_0 +\sup_{[0,t]} \bigl\|\frac{c_1}{\rho_0}\bigl[\frac{\rho_0^2\partial_t \v'\ \v'}{\tilde\eta'^4}\bigr]'\bigr\|_0 +\sup_{[0,t]} \bigl\|\frac{c_2}{\rho_0} \bigl[\frac{\rho_0^2 \v'^3}{\tilde\eta'^5}\bigr]'\bigr\|_0\n\\
 &+\sup_{[0,t]} \bigl\|\frac{2}{\rho_0}\bigl[{\rho_0^2\partial_t^2 \v'}\bigr]'\|_0\|3\int_0^\cdot\frac{\v'}{\tilde\eta'^4}\|_{L^\infty}+6\sup_{[0,t]} \bigl\|\rho_0\partial_t^2\v'\frac{\tilde\eta''}{\tilde\eta'^4}\bigr\|_0\ .
\end{align}
We next estimate each term on the right hand side of (\ref{a9}). For the first term, we will use our estimate (\ref{a8}) from which we infer for each $t\in [0,T_\kappa]$:
\begin{equation*}
\int_I \rho_0^2 [\partial_t^4\v^2+\partial_t^4\v'^2] (t) \le M_0+ C t \ P(\sup_{[0,t]} E).
\end{equation*}
Note that the first term on the left-hand side of (\ref{a10}) comes from the first term of (\ref{a8}), together with the fact that $\d\partial_t^4 v(t,x)=v_4(x)+\int_0^t \partial_t^5 v(\cdot,x).$ Therefore, the Sobolev weighted embedding estimate (\ref{w-embed}) provides us with
the following estimate:
\begin{equation}
\label{a10}
\int_I \partial_t^4\v^2 (t) \le M_0+ C t \ P(\sup_{[0,t]} E).
\end{equation}
The remaining terms will be estimated by simply using the definition of the energy function $E$.  For the second term on the right-hand side of (\ref{a9}), we have that
\begin{align*}
\bigl\|\frac{1}{\rho_0}\bigl[\frac{\rho_0^2\partial_t \v'\ \v'}{\tilde\eta'^4}\bigr]'\bigr\|_0
& \le \|(\rho_0 v_t')'\|_0 \bigl\|\frac{\v'}{\tilde\eta'^4}\bigr\|_{L^\infty}+\bigl\|\v_t'\bigl[\frac{\rho_0\ \v'}{\tilde\eta'^4}\bigr]'\bigr\|_0\\
 & \le C \|(\rho_0 \v_t')'\|_0 \|\v'\|_{\frac{3}{4}}+\bigl\|\v_t'\bigl[\frac{\rho_0'\ \v'}{\tilde\eta'^4}\bigr]\bigr\|_0+\bigl\|\v_t'\bigl[\frac{\rho_0\ \v''}{\tilde\eta'^4}\bigr]\bigr\|_0+4\bigl\|\v_t'\bigl[\frac{\rho_0\ \v'\tilde\eta''}{\tilde\eta'^5}\bigr]\bigr\|_0\\
&   \le C \|(\rho_0 v_1')'+\int_0^\cdot (\rho_0 v_{tt}')'\|_0 \|\v'\|_1^{\frac{1}{2}} \|v_1'+\int_0^\cdot \v_t'\|_{\frac{1}{2}}^{\frac{1}{2}}+C\|v_1'+\int_0^\cdot \v_{tt}'\|_0 \|\v'\|_{\frac{3}{4}}\\
& \qquad \ \ +C \|\v_t'\|_0 \|\rho_0 v''\|_{\frac{3}{4}} + C \|\v'\|_{\frac{3}{4}}\|\int_0^\cdot \v''\|_0 \|\rho_0 v_1'+\int_0^\cdot\rho_0 \v_{tt}'\|_1\\
& \le \|(\rho_0 v_1')'+\int_0^\cdot (\rho_0 v_{tt}')'\|_0 \|\v'\|_1^{\frac{1}{2}} \|v_1' +\int_0^\cdot \v_t'\|_{\frac{1}{2}}^{\frac{1}{2}}\\
& \qquad \ \ +C\|v_1'+\int_0^\cdot \v_{tt}'\|_0 \|\v'\|_1^{\frac{1}{2}} \|v_1'+\int_0^\cdot \v_t'\|_{\frac{1}{2}}^{\frac{1}{2}} \\
& \qquad \ \ +C \|\v_t'\|_0 \|\rho_0 v''\|_1^{\frac{3}{4}} \|\rho_0 v_0''+\int_0^\cdot \rho_0 v_t''\|_0^{\frac{1}{4}},
\end{align*}
where we have used the fact that $\|\cdot\|_{L^\infty}\le C \|\cdot\|_{\frac{3}{4}}$.
Thanks to the definition of $E$, the previous inequality provides us, for any $t\in [0,T_\kappa]$, with
\begin{equation}
\label{a11}
\sup_{[0,t]}\bigl\|\frac{1}{\rho_0}\bigl[\frac{\rho_0^2\partial_t \v'\ \v'}{\tilde\eta'^4}\bigr]'\bigr\|_0\le C \sup_{[0,t]} E^{\frac{3}{4}}\ (M_0+t P(\sup_{[0,t]}E)).
\end{equation}
For the third term on the right-hand side of (\ref{a9}), we have similarly that
\begin{align*}
\bigl\|\frac{1}{\rho_0} \bigl[\frac{\rho_0^2 \v'^3}{\tilde\eta'^5}\bigr]'\bigr\|_0 
&\le \|(\rho_0 v')'\|_{L^\infty} \bigl\|\frac{\v'^2}{\tilde\eta'^5}\bigr\|_{L^2}+\bigl\|\v'\bigl[\frac{\rho_0\ \v'^2}{\tilde\eta'^5}\bigr]'\bigr\|_0\\
& \le C \|(\rho_0 \v')'\|_{\frac{3}{4}} \|\v'\|^2_{\frac{1}{2}}+\bigl\|\v'\bigl[\frac{\rho_0'\ \v'^2}{\tilde\eta'^5}\bigr]\bigr\|_0+2\bigl\|\v'\bigl[\frac{\rho_0\ \v''\v'}{\tilde\eta'^5}\bigr]\bigr\|_0+5\bigl\|\v'\bigl[\frac{\rho_0\ \v'^2\tilde\eta''}{\tilde\eta'^6}\bigr]\bigr\|_0\\
& \le C \|(\rho_0 v_0')'+\int_0^\cdot (\rho_0 v_{t}')'\|_0^{\frac{1}{4}} \|(\rho_0 \v')'\|_1^{\frac{3}{4}} \|\v_0'+\int_0^\cdot \v'_t\|_{\frac{1}{2}}^2 
+C \|\v_0'+\int_0^\cdot \v'_t\|_{\frac{1}{2}}^3\\
&\qquad \ \ +C \|\v'\|_{\frac{1}{2}}^2 \|\rho_0 v''\|_{L^4} + C \|\v'\|_{\frac{1}{2}}^3\|\rho_0 \tilde\eta''\|_{L^4}\\
& \le C \|(\rho_0 v_0')'+\int_0^\cdot (\rho_0 v_{t}')'\|_0^{\frac{1}{4}} \|(\rho_0 \v')'\|_1^{\frac{3}{4}} \|\v_0'+\int_0^\cdot \v'_t\|_{\frac{1}{2}}^2 
+C \|\v_0'+\int_0^\cdot \v'_t\|_{\frac{1}{2}}^3\\
&\qquad \ \ +C \|\v_0'+\int_0^\cdot \v_t'\|_{\frac{1}{2}}^2 \|\rho_0 v_0''+\int_0^\cdot \rho_0 v_t''\|_0^{\frac{1}{2}}\|\rho_0 \v''\|_1^{\frac{1}{2}}  + C \|\v_0'+\int_0^\cdot \v_t'\|_{\frac{1}{2}}^3 \|\int_0^\cdot \rho_0 v''\|_1,
\end{align*}
where we have used the fact that $\|\cdot\|_{L^p}\le C_p \|\cdot\|_{\frac{1}{2}}$, for all $1<p<\infty$.
Again, using the  definition of $E$, the previous inequality provides us for any $t\in [0,T_\kappa]$ with
\begin{equation}
\label{a12}
\sup_{[0,t]}\bigl\| \frac{1}{\rho_0} \bigl[\frac{\rho_0^2 \v'^3}{\tilde\eta'^5}\bigr]' \bigr\|_0\le C \sup_{[0,t]} E^{\frac{1}{2}}\ (M_0+t P(\sup_{[0,t]}E)).
\end{equation}
For the fourth term of the right-hand side of (\ref{a9}), we see that
\begin{align}
\label{a13}
\bigl\|\frac{2}{\rho_0}\bigl[{\rho_0^2\partial_t^2 \v'}\bigr]'\|_0\|3\int_0^\cdot\frac{\v'}{\tilde\eta'^4}\|_{L^\infty} (t) & \le C [\ \|\rho_0 \partial_t^2 v''\|_0+\|\partial_t v'\|_0\ ] t \sup_{[0,t]} \|\v\|_2\n\\
 &\le C t P(\sup_{[0,t]} E)\ .
\end{align}
Similarly, the fifth term of the right-hand side of (\ref{a9}) yields the following estimate:
\begin{align}
\label{a14}
\bigl\|\rho_0\partial_t^2\v'\frac{\tilde\eta''}{\tilde\eta'^4}\bigr\|_0 (t)  &\le C \|\rho_0\partial_t^2\v'\|_{L^\infty}\|\tilde\eta''\|_0\n\\
&\le C \|\rho_0\partial_t^2\v'\|_1 \|\int_0^\cdot\v''\|_0\n\\
& \le C t P(\sup_{[0,t]} E)\ .
\end{align}
Combining the estimates (\ref{a11})--(\ref{a14}), we obtain the inequality
\begin{equation}
\label{a15}
\sup_{[0,t]} \bigl\|\frac{2}{\rho_0}\bigl[{\rho_0^2\partial_t^2 \v'}\bigr]'\bigr\|_0 \le C t P(\sup_{[0,t]} E) + C \sup_{[0,t]} E^{\frac{3}{4}}\ (M_0+t P(\sup_{[0,t]}E)).
\end{equation}

At this stage, we remind the reader,  that the solution $\v$ to our parabolic $ \kappa $-problem is in $X_{T_ \kappa }$, so that for any $t\in [0,T_\kappa]$,  $\partial_t^2 \v\in H^2(I)$.    Notice that
$$
\frac{1}{\rho_0}\bigl[{\rho_0^2\partial_t^2 \v'}\bigr]' =  \rho_0 \partial_t^2 \v'' + 2 \rho_0'\partial_t^2\v' \,,
$$
so (\ref{a15}) is equivalent to 
\begin{equation}\label{a15b}
\sup_{[0,t]} \bigl\|  \rho_0 \partial_t^2 \v'' + 2 \rho_0'\partial_t^2\v' \bigr\|_0 \le C t P(\sup_{[0,t]} E) + C \sup_{[0,t]} E^{\frac{3}{4}}\ (M_0+t P(\sup_{[0,t]}E)).
\end{equation} 
From this inequality, we would like to conclude that
both $\| \partial_t^2 \v '\|_0$  and $\|  \rho_0 \partial_t^2 \v'' \|_0$ are  bounded by the right-hand side of (\ref{a15b}); the regularity
provided by solutions of the $ \kappa $-problem allow us to arrive at this conclusion.

  By expanding the left-hand side of (\ref{a15b}), we see that
\begin{align}
\label{a16}
\sup_{[0,t]} \bigl\|  \rho_0 \partial_t^2 \v'' + 2\rho_0' \partial_t^2\v' \bigr\|_0^2&=\|\rho_0\partial_t^2\v''\|_0^2+4\|\rho_0'\partial_t^2\v'\|_0^2+4\int_I \rho_0  \partial_t^2 \v'' \  \rho_0'\partial_t^2\v'        \,.
\end{align}
Given the regularity of $\partial_t^2 \v$ provide by our parabolic $ \kappa$-problem, we notice that the cross-term in (\ref{a16}) is an exact derivative, 
$$
4\int_I \rho_0  \partial_t^2 \v'' \  \rho_0'\partial_t^2\v'  = 2 \int_I \rho_0\rho_0' \frac{\partial}{\partial x} |\partial_t^2 \v'|^2 \,,
$$
so that by integrating-by-parts, we find that
$$
4\int_I \rho_0  \partial_t^2 \v'' \  \rho_0'\partial_t^2\v'  = -2 \|\rho_0'\partial_t^2\v'\|_0^2  - \int_I \rho_0 \partial_t^2
\v' \  \rho_0'' \partial_t^2 \v' \,,
$$
and hence (\ref{a16}) becomes
\begin{align}
\label{a16b}
\sup_{[0,t]} \bigl\|  \rho_0 \partial_t^2 \v'' + 2\rho_0' \partial_t^2\v' \bigr\|_0^2&=\|\rho_0\partial_t^2\v''\|_0^2+2\|\rho_0'\partial_t^2\v'\|_0^2 - \int_I \rho_0 \partial_t^2
\v' \  \rho_0'' \partial_t^2 \v'        \,.
\end{align}
Since the energy function $E$ contains $\rho_0 \partial_t^2 \v(t) \in H^2(I)$ and $\partial_t^2\v(t) \in H^1(I)$, the
fundamental theorem of calculus shows that
$$
\int_I \rho_0 \partial_t^2\v' \  \rho_0'' \partial_t^2 \v'   \le
C t P(\sup_{[0,t]} E) + C \sup_{[0,t]} E^{\frac{3}{4}}\ (M_0+t P(\sup_{[0,t]}E)) \,.
$$
Combining this inequality with (\ref{a16b}) and  (\ref{a15}), yields 
\begin{equation*}
\sup_{[0,t]} [ \|\rho_0\partial_t^2\v''\|_0+\|\rho_0'\partial_t^2\v'\|_0 ]  \le C t P(\sup_{[0,t]} E) + C \sup_{[0,t]} E^{\frac{3}{4}}\ (M_0+t P(\sup_{[0,t]}E)),
\end{equation*}
and thus
\begin{equation*}
\sup_{[0,t]} [ \|\rho_0\partial_t^2\v''\|_0+\|\rho_0'\partial_t^2\v'\|_0 + \|\rho_0\partial_t^2\v'\|_0 ]  \le M_0+C t P(\sup_{[0,t]} E) + C \sup_{[0,t]} E^{\frac{3}{4}}\ (M_0+t P(\sup_{[0,t]}E)),
\end{equation*}
and hence with the physical vacuum conditions on $\rho_0$ given by (\ref{degen1}) and (\ref{degen2}), we have that
\begin{equation*}
\sup_{[0,t]} [ \|\rho_0\partial_t^2\v''\|_0+\|\partial_t^2\v'\|_0  ]  \le M_0+C t P(\sup_{[0,t]} E) + C \sup_{[0,t]} E^{\frac{3}{4}}\ (M_0+t P(\sup_{[0,t]}E)),
\end{equation*}
which, together with (\ref{a10}), provide us with the estimate
\begin{equation}
\label{a17}
\sup_{[0,t]} [ \|\rho_0\partial_t^2\v''\|_0+\|\partial_t^2\v\|_1  ]  \le M_0+C t P(\sup_{[0,t]} E) + C \sup_{[0,t]} E^{\frac{3}{4}}\ (M_0+t P(\sup_{[0,t]}E)).
\end{equation}

\subsection{Elliptic and Hardy-type estimates for $v(t)$.}  Having obtained the estimates for $\partial_t^2 \v(t)$ in (\ref{a17}), we can
next obtain our estimates for $\v(t)$.
To do so, we consider the first time-differentiated version of (\ref{approximate.a}), which yields the equation
\begin{equation*}
 -2\bigl[\frac{\rho_0^2 \v'}{\tilde\eta'^3}\bigr]'-\kappa [\rho_0^2 \partial_t \v']'=-\rho_0\partial_t^2 \v,
\end{equation*}
which we rewrite as the following identity:
\begin{align}
 -\frac{2}{\rho_0}\bigl[{\rho_0^2\v'}\bigr]'-\frac{\kappa}{\rho_0} [\rho_0^2 \partial_t \v']'=&-\partial_t^2 \v-\frac{2}{\rho_0}\bigl[{\rho_0^2 \v'}\bigr]'(1-\frac{1}{\tilde\eta'^3})\ . \label{cs11}
\end{align}
Using Lemma \ref{kelliptic}, we see that  for any $t\in [0,T_\kappa]$,
\begin{align}
\label{a18}
\sup_{[0,t]} \bigl\|\frac{1}{\rho_0}\bigl[{\rho_0^2 \v'}\bigr]'\bigr\|_1 \le & C\sup_{[0,t]}\| \partial_t^2 \v\|_1  +C\sup_{[0,t]} \bigl\|\frac{1}{\rho_0}\bigl[{\rho_0^2 \v'}\bigr]'\|_1\|\int_0^\cdot\frac{\v'}{\tilde\eta'^4}\|_{L^\infty} +C \|\frac{1}{\rho_0}[\rho_0^2 \v']'.\frac{\tilde\eta''}{\tilde\eta'^4}\|_0
\end{align}
We next estimate each term on the right hand side of (\ref{a18}). The bound for the first term on the right-hand side of
(\ref{a18}) is provided by (\ref{a17}).

The second term of the right-hand side of (\ref{a18}) is estimated as follows:
\begin{align}
\label{a20}
\bigl\|\frac{1}{\rho_0}\bigl[{\rho_0^2 \v'}\bigr]'\|_1\|\int_0^\cdot\frac{\v'}{\tilde\eta'^4}\|_{L^\infty} (t)  &\le C [ \| \rho_0 \v'''\|_0+\|v\|_2]\ t \sup_{[0,t]} \|\v\|_2\n\\
 &\le C t P(\sup_{[0,t]} E)\ .
\end{align}

For the third term on the right-hand side of (\ref{a18}),
\begin{align}
\label{a21}
\|\frac{1}{\rho_0}[\rho_0^2 \v']'.[\frac{\tilde\eta''}{\tilde\eta'^4}]\|_1 (t)
& \le C \|\frac{1}{\rho_0}[\rho_0^2 \v']'\|_{L^\infty} \|\tilde\eta''\|_0\n\\
 &\le C [ \| \rho_0 \v'''\|_0+\|v\|_2] \|\int_0^t \v''\|\n\\
& \le C t P(\sup_{[0,t]} E)\ .
\end{align}
Combining these estimates provides the inequality
\begin{equation*}
\sup_{[0,t]} \bigl\|\frac{1}{\rho_0}\bigl[{\rho_0^2\v'}\bigr]'\bigr\|_1 \le C t P(\sup_{[0,t]} E) + C \sup_{[0,t]} E^{\frac{3}{4}}\ (M_0+t P(\sup_{[0,t]}E)),
\end{equation*}
which leads us immediately to:
\begin{equation}
\label{a22}
\sup_{[0,t]} \| \rho_0\v'''+3\rho_0'v''\|_0 \le C t P(\sup_{[0,t]} E) + C \sup_{[0,t]} E^{\frac{3}{4}}\ (M_0+t P(\sup_{[0,t]}E)) \,.
\end{equation}
Now, since for any $t\in [0,T_\kappa]$, the solution $\v$ to our parabolic $ \kappa $-problem is in $H^3(I)$,  we infer that $\rho_0 v'''\in L^2(I)$. We can then apply the same integration-by-parts argument as in \cite{CoLiSh2009} to find that
\begin{align}
\label{a23}
\|\rho_0\v'''+3\rho_0'v''\|^2_0&=\|\rho_0\v'''\|_0^2+9\|{\rho_0'}\v''\|_0^2+3\int_I \rho_0\rho_0' [|\v''|^2]'\n\\
&=\|\rho_0\v'''\|_0^2+9\|\rho_0'\v''\|_0^2-3\int_I [\rho_0\rho_0''+\rho_0'^2] |\v''|^2\n\\
&=\|\rho_0\v'''\|_0^2+6\|\rho_0'\v''\|_0^2-3\int_I \rho_0\rho_0'' |\v''|^2
\end{align}
Combined with (\ref{a22}), this yields:
\begin{align*}
\sup_{[0,t]} [ \|\rho_0\v'''\|_0+\|\rho_0'\v''\|_0 ]  \le & C t P(\sup_{[0,t]} E) + C \sup_{[0,t]} E^{\frac{3}{4}}\ (M_0+t P(\sup_{[0,t]}E))\n\\
&+M_0+C\|\int_0^t \sqrt{\rho_0}\v_t''\|_0,
\end{align*}
and thus
\begin{equation*}
\sup_{[0,t]} [ \|\rho_0\v'''\|_0+\|\rho_0'\v''\|_0 + \|\rho_0\v''\|_0 ]  \le M_0+C t P(\sup_{[0,t]} E) + C \sup_{[0,t]} E^{\frac{3}{4}}\ (M_0,
\end{equation*}
With (\ref{degen1}) and (\ref{degen2}),  it follows that
\begin{equation*}
\sup_{[0,t]} [ \|\rho_0\v'''\|_0+\|\v''\|_0  ]  \le M_0+C t P(\sup_{[0,t]} E) + C \sup_{[0,t]} E^{\frac{3}{4}}\ (M_0+t P(\sup_{[0,t]}E)),
\end{equation*}
and hence
\begin{equation}
\label{a24}
\sup_{[0,t]} [ \|\rho_0\v'''\|_0+\|\v\|_2  ]  \le M_0+C t P(\sup_{[0,t]} E) + C \sup_{[0,t]} E^{\frac{3}{4}}\ (M_0+t P(\sup_{[0,t]}E)).
\end{equation}

\subsection{Elliptic and Hardy-type estimates for $\partial_t^3v(t)$ and $\partial_t v(t)$}
We consider  the fourth time-differentiated version of (\ref{approximate.a}):
$$
 \bigl[\partial_t^4\frac{\rho_0^2}{\tilde\eta'^2}\bigr]'-\kappa [\rho_0^2 \partial_t^4 \v']'=-\rho_0 \partial_t^5 \v \,,
$$
which can be rewritten as
\begin{equation*}
 -2\bigl[\frac{\rho_0^2\partial_t^3 \v'}{\tilde\eta'^3}\bigr]'-\kappa [\rho_0^2 \partial_t^4 \v']'=-\rho_0\partial_t^5 \v+c_1\bigl[\frac{\rho_0^2\partial_t^2 \v'\ \v'}{\tilde\eta'^4}\bigr]' +c_2 \bigl[\frac{\rho_0^{\frac{3}{2}} \partial_t\v'^2}{\tilde\eta'^5}\bigr]',
\end{equation*}
for some constants $c_1$ and $c_2$.  By employing the fundamental theorem of calculus and dividing by $\rho_0^ {\frac{1}{2}} $,
we obtain the equation
\begin{align*}
 -\frac{2}{\rho_0^{\frac{1}{2}}}\bigl[{\rho_0^2\partial_t^3 \v'}\bigr]'-\frac{\kappa}{\rho_0^{\frac{1}{2}}} [\rho_0^2 \partial_t^4 \v']'=&-\sqrt{\rho_0}\partial_t^5 \v+\frac{c_1}{{\rho_0}^{\frac{1}{2}}}\bigl[\frac{\rho_0^2\partial_t^2 \v'\ \v'}{\tilde\eta'^4}\bigr]' +\frac{c_2}{\rho_0^{\frac{1}{2}}} \bigl[\frac{\rho_0^2 \partial_t\v'^2}{\tilde\eta'^5}\bigr]'\\
 &-\frac{2}{\rho_0^{\frac{1}{2}}}\bigl[{\rho_0^2\partial_t^3 \v'}\bigr]'(1-\frac{1}{\tilde\eta'^3})-6\rho_0^{\frac{3}{2}}\partial_t^3\v'\frac{\tilde\eta''}{\tilde\eta'^4}\ .
\end{align*}
For any $t\in [0,T_\kappa]$,  Lemma \ref{kelliptic} provides the $ \kappa $-independent estimate
\begin{align}
\label{a25}
\sup_{[0,t]} \bigl\|\frac{2}{\rho_0^{\frac{1}{2}}}\bigl[{\rho_0^2\partial_t^3 \v'}\bigr]'\bigr\|_0 \le & \sup_{[0,t]}\|\sqrt{\rho_0}\partial_t^5 \v\|_0 +\sup_{[0,t]} \bigl\|\frac{c_1}{\rho_0^{\frac{1}{2}}}\bigl[\frac{\rho_0^2\partial_t^2 \v'\ \v'}{\tilde\eta'^4}\bigr]'\bigr\|_0 +\sup_{[0,t]} \bigl\|\frac{c_2}{\rho_0^{\frac{1}{2}}} \bigl[\frac{\rho_0^2 \partial_t\v'^2}{\tilde\eta'^5}\bigr]'\bigr\|_0\n\\
 &+\sup_{[0,t]} \bigl\|\frac{2}{\rho_0^{\frac{1}{2}}}\bigl[{\rho_0^2\partial_t^3 \v'}\bigr]'\|_0\|3\int_0^\cdot\frac{\v'}{\tilde\eta'^4}\|_{L^\infty}+6\sup_{[0,t]} \bigl\|\rho_0^{\frac{3}{2}}\partial_t^2\v'\frac{\tilde\eta''}{\tilde\eta'^4}\bigr\|_0\ .
\end{align}

We estimate each term on the right-hand side of (\ref{a25}).  The first term on the right-hand side is bounded by
$ M_0+ C t \ P(\sup_{[0,t]} E)$ thanks to (\ref{a8}).
 For the second term on the right-hand side of (\ref{a25}), we have that
\begin{align*}
\bigl\|\frac{1}{\rho_0^{\frac{1}{2}}}\bigl[\frac{\rho_0^2\partial_t^2 \v'\ \v'}{\tilde\eta'^4}\bigr]'\bigr\|_0 &\le \|\sqrt{\rho_0}(\rho_0 \partial_t^2\v')'\|_0 \bigl\|\frac{\v'}{\tilde\eta'^4}\bigr\|_{L^\infty}+\bigl\|\sqrt{\rho_0}\partial_t^2\v'\bigl[\frac{\rho_0\ \v'}{\tilde\eta'^4}\bigr]'\bigr\|_0\\
 &\le C \|\sqrt{\rho_0}(\rho_0 \partial_t^2\v')'\|_0 \|\v'\|_{\frac{3}{4}}+\bigl\|\sqrt{\rho_0}\partial_t^2\v'\bigl[\frac{\rho_0'\ \v'}{\tilde\eta'^4}\bigr]\bigr\|_0+\bigl\|\sqrt{\rho_0}\partial_t^2\v'\bigl[\frac{\rho_0\ \v''}{\tilde\eta'^4}\bigr]\bigr\|_0\n\\
&\ \ +4\bigl\|\sqrt{\rho_0}\partial_t^2\v'\bigl[\frac{\rho_0\ \v'\tilde\eta''}{\tilde\eta'^5}\bigr]\bigr\|_0\\
& \le C \|\sqrt{\rho_0}(\rho_0 \v_2')'+\int_0^\cdot \sqrt{\rho_0}(\rho_0 \partial_t^3\v')'\|_0 \|\v'\|_1^{\frac{1}{2}} \|\v_1'+\int_0^\cdot \v_t'\|_{\frac{1}{2}}^{\frac{1}{2}}\n\\
&\ \ +C\|\sqrt{\rho_0}\v_2'+\int_0^\cdot \sqrt{\rho_0}\partial_t^3\v'\|_0 \|\v'\|_{\frac{3}{4}}\\
&\ \ +C \|\sqrt{\rho_0}\partial_t^2\v'\|_0 \|\rho_0 v''\|_{\frac{3}{4}} + C \|\v'\|_{\frac{3}{4}}\|\int_0^\cdot \v''\|_0 \|\rho_0^{\frac{3}{2}} \v_2'+\int_0^\cdot\rho_0^{\frac{3}{2}} \partial_t^3\v'\|_1\\
& \le \|\sqrt{\rho_0}(\rho_0 \v_2')'+\int_0^\cdot \sqrt{\rho_0}(\rho_0 \v_{tt}')'\|_0 \|\v'\|_1^{\frac{1}{2}} \|\v_1' +\int_0^\cdot \v_t'\|_{\frac{1}{2}}^{\frac{1}{2}}\\
&\ \ +C\|\sqrt{\rho_0}\v_2'+\int_0^\cdot \sqrt{\rho_0}\v_{tt}'\|_0 \|\v'\|_1^{\frac{1}{2}} \|\v_1'+\int_0^\cdot \v_t'\|_{\frac{1}{2}}^{\frac{1}{2}} \\
&\ \ +C \|\sqrt{\rho_0}\partial_t^2\v'\|_0 \|\rho_0 v''\|_1^{\frac{3}{4}} \|\rho_0 v_0''+\int_0^\cdot \rho_0 v_t''\|_0^{\frac{1}{4}}\n\\
&\ \ +C \|\v'\|_1^{\frac{1}{2}} \|\v_1'+\int_0^\cdot \v_t'\|_{\frac{1}{2}}^{\frac{1}{2}} \|\int_0^\cdot \v''\|_0 \|\rho_0^{\frac{3}{2}} \v_2'+\int_0^\cdot\rho_0^{\frac{3}{2}} \partial_t^3\v'\|_1 \,,
\end{align*}
where we have again used the fact that $\|\cdot\|_{L^\infty}\le C \|\cdot\|_{\frac{3}{4}}$.
Thanks to the definition of $E$, the previous inequality shows that for any $t\in [0,T_\kappa]$, 
\begin{equation}
\label{a27}
\sup_{[0,t]}\bigl\|\frac{1}{\sqrt{\rho_0}}\bigl[\frac{\rho_0^2\partial_t^2 \v'\ \v'}{\tilde\eta'^4}\bigr]'\bigr\|_0\le C \sup_{[0,t]} E^{\frac{3}{4}}\ (M_0+t P(\sup_{[0,t]}E)).
\end{equation}
For the third term on the right-hand side of (\ref{a25}), we have similarly that
\begin{align}
\label{a28}
\bigl\|\frac{1}{\sqrt{\rho_0}} \bigl[\frac{\rho_0^2 \partial_t\v'^2}{\tilde\eta'^5}\bigr]'\bigr\|_0 (t) & \le 2\|(\rho_0 \partial_t\v')'\|_{0} \bigl\|\sqrt{\rho_0}\frac{\partial_t\v'}{\tilde\eta'^5}\bigr\|_{L^\infty}+5\bigl\|\sqrt{\rho_0}\frac{\partial_t\v'^2}{\tilde\eta'^6}\|_{L^1}\|\rho_0 \tilde\eta''\|_{L^\infty}\n\\
& \le C \|(\rho_0\v_1')'+\int_0^t (\rho_0\partial_{tt}\v')'\|_0\|\sqrt{\rho_0}\partial_t\v'\|_{\frac{3}{4}}+C \|\partial_t\v'\|_0^2\|\int_0^t (\rho_0\v'')'\|_0\n\\
& \le C \|(\rho_0\v_1')'+\int_0^t (\rho_0\partial_{tt}\v')'\|_0\|\sqrt{\rho_0} \partial_t\v'\|_{0}^{1-\alpha}\|(\sqrt{\rho_0}\partial_t\v')'\|_{L^{2-a}}^{\alpha}\n\\
&\ \ +C \|\partial_t\v'\|_0^2\|\int_0^t (\rho_0\v'')'\|_0\n\\
& \le C \|(\rho_0\v_1')'+\int_0^t (\rho_0\partial_{tt}\v')'\|_0\|\v'_1+\int_0^t\partial_{tt}\v'\|_{0}^{1-\alpha}\|(\sqrt{\rho_0}\partial_t\v')'\|_{L^{2-a}}^{\alpha}\n\\
&\ \ +C \|\partial_t\v'\|_0^2\|\int_0^t (\rho_0\v'')'\|_0 \,,
\end{align}
where $0<a<\frac{1}{2}$ is given and $0<\alpha=\frac{3-3a}{4+3a}<1$.

The only term on the right-hand side of (\ref{a28}) which is not directly contained in the definition of $E$ is $\|(\sqrt{\rho_0}\partial_t\v')'\|_{L^{2-a}}^\alpha$. To this end, we notice that
\begin{align}
\label{a29}
\|(\sqrt{\rho_0}\partial_t\v')'\|_{L^{2-a}} &\le \|\frac{\partial_t\v'}{2\sqrt{\rho_0}}\|_{L^{2-a}}+\|\sqrt{\rho_0} \v_{tt}''\|_0\n\\
 &\le \|\frac{1}{2\sqrt{\rho_0}}\|_{L^{2-\frac{a}{2}}}\|\partial_t\v'\|_{\frac{1}{2}}+\|\sqrt{\rho_0} \v_{tt}''\|_0 \,,
\end{align}
where we have used the fact that $\|\cdot\|_{L^p}\le C_p \|\cdot\|_{\frac{1}{2}}$, for all $1<p<\infty$.
Thanks to the definition of $E$, the previous inequality and (\ref{a28}) provides us for any $t\in [0,T_\kappa]$ with
\begin{equation}
\label{a30}
\sup_{[0,t]}\bigl\| \frac{1}{\rho_0} \bigl[\frac{\rho_0^2 \v'^3}{\tilde\eta'^5}\bigr]' \bigr\|_0\le C \sup_{[0,t]} E^\alpha\ (M_0+t P(\sup_{[0,t]}E)),
\end{equation}
where we recall again that $0<\alpha=\frac{3-3a}{4+3a}<1$.

The fourth term on the right-hand side of (\ref{a25}) is easily treated:
\begin{align}
\label{a31}
\bigl\|\frac{1}{\sqrt{\rho_0}}\bigl[{\rho_0^2\partial_t^3 \v'}\bigr]'\|_0\|\int_0^t\frac{\v'}{\tilde\eta'^4}\|_{L^\infty} (t)  & \le C [\ \|\rho_0^{\frac{3}{2}} \partial_t^3 v''\|_0+\|\sqrt{\rho_0}\partial_t^3 v'\|_0\ ] t \sup_{[0,t]} \|\v\|_2\n\\
 & \le C t P(\sup_{[0,t]} E)\ .
\end{align}
Similarly, the fifth term on the right-hand side of (\ref{a25}) is estimated as follows:
\begin{align}
\label{a32}
\bigl\|\rho_0^{\frac{3}{2}}\partial_t^3\v'\frac{\tilde\eta''}{\tilde\eta'^4}\bigr\|_0 (t) & \le C \|\rho_0^{\frac{3}{2}}\partial_t^3\v'\|_{L^\infty}\|\tilde\eta''\|_0\n\\
& \le C [\ \|\rho_0^{\frac{3}{2}} \partial_t^3 v''\|_0+\|\sqrt{\rho_0}\partial_t^3 v'\|_0\ ] \bigl\|\int_0^t\v''\bigr\|_0\n\\
& \le C t P(\sup_{[0,t]} E)\ .
\end{align}
Combining the estimates (\ref{a25})--(\ref{a32}), we can infer that
\begin{equation}
\label{a33}
\sup_{[0,t]} \bigl\|\frac{1}{\rho_0^{\frac{1}{2}}}\bigl[{\rho_0^2\partial_t^3 \v'}\bigr]'\bigr\|_0 \le C t P(\sup_{[0,t]} E) + C \sup_{[0,t]} E^\alpha\ (M_0+t P(\sup_{[0,t]}E)).
\end{equation}
Now, since for any $t\in [0,T_\kappa]$, solutions to our parabolic $ \kappa $-problem have the regularity $\partial_t^2v\in H^2(I)$, 
we integrate-by-parts:
\begin{align}
\label{a34}
\bigl\|\frac{1}{\rho_0^{\frac{1}{2}}}\bigl[{\rho_0^2\partial_t^3 \v'}\bigr]'\bigr\|^2_0&=\|\rho_0^{\frac{3}{2}}\partial_t^3\v''\|_0^2+4\|{\rho_0}^{\frac{1}{2}}\rho_0'\partial_t^3\v'\|_0^2+2\int_I \rho_0'\rho_0^2 [|\partial_t^3\v'|^2]'\n\\
&=\|\rho_0^{\frac{3}{2}}\partial_t^3\v''\|_0^2+4\|\rho_0'\rho_0^{\frac{1}{2}}\partial_t^3\v'\|_0^2-4\int_I \rho_0'^2\rho_0 |\partial_t^3\v'|^2-2\int_I \rho_0''\rho_0^2 |\partial_t^3\v'|^2\n\\
&=\|\rho_0^{\frac{3}{2}}\partial_t^3\v''\|_0^2-2 \int_I \rho_0''\rho_0^2 |\partial_t^3\v'|^2.
\end{align}
Combined with (\ref{a33}), and the fact that $\rho_0\partial_t^3\v'=\rho_0\v_3+\int_0^cdot \rho_0\partial_t^4\v'$ for the second term on the right-hand side of (\ref{a34}), we find that
\begin{equation}
\label{a35}
\sup_{[0,t]} \|\rho_0^{\frac{3}{2}}\partial_t^3\v''\|_0   \le C t P(\sup_{[0,t]} E) + C \sup_{[0,t]} E^\alpha\ (M_0+t P(\sup_{[0,t]}E)).
\end{equation}
Now, since $\frac{1}{\rho_0^{\frac{1}{2}}}\bigl[{\rho_0^2\partial_t^3 \v'}\bigr]'=\rho_0^{\frac{3}{2}}\partial_t^3\v''+2 {\rho_0}^{\frac{1}{2}}\rho_0'\partial_t^3\v'$, the estimates (\ref{a33}) and (\ref{a35}) also imply that
\begin{equation}
\label{a36}
\sup_{[0,t]} \|{\rho_0}^{\frac{1}{2}}\rho_0'\partial_t^3\v'\|_0   \le C t P(\sup_{[0,t]} E) + C \sup_{[0,t]} E^\alpha\ (M_0+t P(\sup_{[0,t]}E)).
\end{equation}
Therefore,
\begin{align*}
\sup_{[0,t]} [ \|\rho_0^{\frac{3}{2}}\partial_t^3\v''\|_0+\|\rho_0'\rho_0^{\frac{1}{2}}\partial_t^3\v'\|_0  +\|\rho_0^{\frac{3}{2}}\partial_t^3\v'\|_0]  \le &M_0+C t P(\sup_{[0,t]} E) \\
&+ C \sup_{[0,t]} E^\alpha\ (M_0+t P(\sup_{[0,t]}E)),
\end{align*}
so that with (\ref{degen1}) and (\ref{degen2}), 
\begin{equation*}
\sup_{[0,t]} [ \|\rho_0^{\frac{3}{2}}\partial_t^3\v''\|_0+\|\rho_0^{\frac{1}{2}}\partial_t^3\v'\|_0  ]  \le M_0+C t P(\sup_{[0,t]} E) + C \sup_{[0,t]} E^\alpha\ (M_0+t P(\sup_{[0,t]}E)) \,.
\end{equation*}
Together with (\ref{a10}) and the  weighted embedding estimate (\ref{w-embed}), the above inequality shows that
\begin{equation}
\label{a37}
\sup_{[0,t]} [ \|\rho_0^{\frac{3}{2}}\partial_t^3\v''\|_0+\|\partial_t^3\v\|_{\frac{1}{2}}  ]  \le M_0+C t P(\sup_{[0,t]} E) + C \sup_{[0,t]} E^\alpha\ (M_0+t P(\sup_{[0,t]}E)).
\end{equation}

By studying the second  time-differentiated version of (\ref{approximate.a}) in the same manner, we find that
\begin{equation}
\label{a38}
\sup_{[0,t]} [ \|\rho_0^{\frac{3}{2}}\partial_t\v'''\|_0+\|\partial_t\v\|_{\frac{3}{2}}  ]  \le M_0+C t P(\sup_{[0,t]} E) + C \sup_{[0,t]} E^\alpha\ (M_0+t P(\sup_{[0,t]}E)) \,.
\end{equation}

\section{Proof of Theorem \ref{theorem_main}}
 
\subsection{Time of existence and bounds independent of $ \epsilon $ and existence of solutions to (\ref{ce0})}\label{subsec_finish1}
 Summing the inequalities (\ref{a8}), (\ref{a17}), (\ref{a24}), (\ref{a37}), (\ref{a38}), we find that
$$
 \sup_{t \in [0,T]} E(t) \le
 M_0+C t P(\sup_{[0,t]} E) + C \sup_{[0,t]} E^\alpha\ (M_0+t P(\sup_{[0,t]}E)) \,.
 $$
 As $ \alpha < 1$, by employing Young's inequality and readjusting the constants, we obtain
$$
 \sup_{t \in [0,T]} E(t) \le
 M_0 +  C \, T\, P({\sup_{t\in[0,T]}} E(t)) \,.
 $$

Just as in Section 9 of \cite{CoSh2006}, this provides us with
a time of existence $T_1$ independent of $\kappa$ and an
estimate on $(0,T_1)$ independent of $\kappa$ of the type:
\begin{equation}
\sup_{t \in [0,T_1]}  E(t) \le 2 M_0 \,. \label{eq420}
\end{equation}
In particular, our sequence of solutions $(v_ \kappa )$ satisfy the $ \kappa $-independent bound (\ref{eq420}) on
the $ \kappa $-independent time-interval $(0,T_1)$.   

\subsection{The limit as $ \kappa \rightarrow 0$} 
By the $\kappa$-independent estimate (\ref{eq420}), there exists a subsequence of
$\{v_ \kappa \}$ which converges weakly to $v$ in $L^2(0,T; H^2(I))$.  With $\eta(t,x) = x + \int_0^t v(s,x)ds$,
by standard compactness arguments, we see that a further subsequence of $v_ \kappa $ and $\eta'_ \kappa$
uniformly converges to $v$ and $\eta'$, respectively, which shows that $v$ is the solution to (\ref{ce0}) and
 $v(0,x) = u_0(x)$.

\subsection{Uniqueness of solutions to the compressible Euler equations (\ref{ce0})} For uniqueness, we require the initial
data to have one space-derivative better regularity than for existence.   Given the assumption  (\ref{uniquedata}) on the data $(u_0,\rho_0)$,
repeating our argument for existence, we can produce a solution $v$ on $[0,T_1]$ which satisfies the estimate
$$
\sum_{s=0}^3\left[ \|\partial_t^{2s}v(t,\cdot)\|^2_{H^{3-s}(I)} +  \|\rho_0 \partial_t^{2s} v(t,\cdot)\|^2_{H^{4-s}(I)} \right] < \infty \,,
$$
and has the flow $\eta(t,x)=x + \int_0^t v(s,x) ds$.
For the sake of contradiction, let us assume that $w$ is also a solution on $[0,T_1]$ with initial data $(u_0, \rho_0)$, satisfying the same estimate, with flow
$\psi(t,x)=x + \int_0^t w(s,x)ds$.

We define
$$
\delta v = v -w \,,
$$
in which case we have the following equation for $ \delta v$:
\begin{subequations}
  \label{cunique}
\begin{alignat}{2}
\rho_0  \delta v_t + (\rho_0^2[{\eta' }^{-2} - {\psi'} ^{-2} ])'&=0  &&\text{in} \ \  I \times (0,T_1] 
\,, \label{cunique.a}\\
\delta v&= 0 \ \ \ &&\text{on} \ \  I \times \{t=0\} \,, \label{cunique.b}\\
\rho_0& = 0 \ \ &&\text{ on }  \partial I \,. \label{cunique.c}
\end{alignat}
\end{subequations}

By considering the fifth time-differentiated version of (\ref{cunique.a}) and taking the $L^2(I)$ inner-product with $ \partial_t \delta v$, we obtain the analogue of (\ref{a8}) (with $ \kappa =0$) for
$ \delta v$.  The additional error terms which arise  are easily controlled
by the fact that both $v$ and $w$ have one space-derivative better regularity than the energy function $E$.    This produces a good
bound for $ \partial_t^4  \delta v \in L^ \infty (0,T_1; L^2(I))$.  By repeating the elliptic and Hardy-type estimates for $\partial_t^2 \delta v \in
L^ \infty (0,T_1; H^1(I))$ and $\partial \delta v \in
L^ \infty (0,T_1; H^2(I))$, and using (\ref{cunique.b}),  we obtain the inequality
\begin{align*} 
& \sup_{t \in [0,T_1]} (\|\partial_t^4  \delta v(t)\|_0^2 + \|\partial_t^2  \delta v(t)\|_1^2 +  \|  \delta v(t)\|_2^2 ) \\
& \qquad\qquad\qquad
\le C\, T_1 \, 
P(  \sup_{t \in [0,T_1]} (\|\partial_t^4  \delta v(t)\|_0^2 + \|\partial_t^2  \delta v(t)\|_1^2 +  \|  \delta v(t)\|_2^2 )) \,,
\end{align*} 
which shows that $ \delta v=0$.

\subsection{Optimal regularity for initial data} We smoothed our initial data  $(u_0,\rho_0)$ in order to  construct solutions to
our degenerate parabolic $ \kappa $-problem (\ref{approximate}).   Having obtained solutions which depend only on $E(0,v)$,
a standard density argument shows that the initial data needs only to satisfy $M_0 <  \infty $.

\section{The case $\gamma \neq 2$}
In this section, we describe the modifications to the energy function and the methodology for the case that $\gamma \neq 2$.  
We denote by $a_0$ the integer satisfying the inequality
$$
1 < 1+ {\frac{1}{\gamma -1}}  -a_0 \le 2 \,.
$$
Letting
$$
d(x) = \text{dist}(x, \partial I) \,,
$$
We consider the following higher-order energy function:
\begin{align*} 
E_\gamma(t,v) & = \sum_{s=0}^4 \| v(t, \cdot )\|^2_{2 - {\frac{s}{2}} }  + \sum_{s=0}^2 \| d \, \partial_t^{2s} v(t, \cdot )\|^2_{3-s} 
+  \| \sqrt{d} \, \partial_t\partial_x^{2} v(t, \cdot )\|^2_{0} +  \| \sqrt{d} \, \partial_t^3 \partial_x v(t, \cdot )\|^2_{0} \\
& \qquad + \sum_{a=0}^{a_0} \| \sqrt{d}^{1+ {\frac{1}{\gamma-1}} - a} \partial_t^{4+a_0-a} v'( t, \cdot )\|_0^2 \,,
\end{align*} 
and define the polynomial function $M_0^\gamma = P( E_ \gamma (0, v))$.
Notice the last sum in $E_ \gamma $ appears whenever $ \gamma < 2$, and the number of time-differentiated problems increases
as $\gamma \to 1$.

Using the same procedure as we have detailed for the case that $\gamma =2$, we have the following 

\begin{theorem}[Existence and uniqueness for any $\gamma >1$]\label{thm_main2}
Given initial data $(u_0, \rho_0)$ such that $M^\gamma_0< \infty $ and the physical vacuum condition (\ref{degen}) holds for $\rho_0$, there exists a  solution to (\ref{ceuler0}) (and hence (\ref{ceuler})) on $[0,T_\gamma]$ for $T_\gamma>0$ taken
 sufficiently small, such that
 $$
 \sup_{t \in [0,T]} E(t) \le 2M^\gamma_0 \,.
 $$
 Moreover if the initial data satisfies
 $$
 \sum_{s=0}^3 \|\partial_t^sv(0,\cdot)\|^2_{H^{3-s}(I)} + \sum_{s=0}^3 \|d\, \partial_t^{2s} v(0,\cdot)\|^2_{H^{4-s}(I)} 
 + \sum_{a=0}^{a_0} \| \sqrt{d}^{1+ {\frac{1}{\gamma-1}} - a} \partial_t^{6+a_0-a} v'( 0, \cdot )\|_0^2< \infty \,,
 $$
 then the solution is unique.
\end{theorem}

\vspace{.1 in}

\noindent
{\bf Acknowledgments.}
SS was supported by the National Science Foundation under
grant DMS-0701056.

\end{document}